\documentclass{article}
\usepackage{bm}
\usepackage{tikz,mathpazo}
\usepackage{pgf}
\usepackage[top=2.5cm, bottom=2.5cm, left=3cm, right=3cm]{geometry}   
\usepackage{indentfirst}
\usepackage{graphicx}
\usepackage{graphics}
\usepackage[toc,page,title,titletoc,header]{appendix}
\usepackage{bm}
\usepackage{bbm}
\usepackage{listings}  
\usepackage{amsmath}
\usepackage{setspace} 
\usepackage{indentfirst}
\usepackage{caption}
\usepackage{multirow} 
\usepackage{lipsum,multicol}
\usepackage{pdfpages}
\usepackage{float}
\usepackage{amsthm}
\usepackage{pdfpages}
\usepackage{url}
\usepackage{colortbl}
\usepackage{subfigure}
\usepackage{epsfig}
\usepackage{epstopdf}
\usetikzlibrary{shapes.geometric, arrows}
\usepackage{fancyhdr}
\usepackage{abstract}
\usepackage{tikz,mathpazo}
\usetikzlibrary{shapes.geometric, arrows}
\usepackage{amssymb}
\usepackage{latexsym}
\usepackage{verbatim}
\usepackage[numbers]{natbib} 
\usepackage{booktabs}

\usepackage{hyperref}
\hypersetup{
    colorlinks=true,
    linkcolor=blue,
    filecolor=magenta,      
    urlcolor=blue,
    citecolor=blue,
    }

\allowdisplaybreaks[4]
\newcommand{\inner}[1]{\left\langle #1 \right\rangle}
\newcommand{\norm}[1]{\left\Vert #1\right\Vert}
\newcommand{\bb}[1]{\mathbb{#1}}

\newcommand{\ca}[1]{\mathcal{#1}}

\newcommand{\Diag}[0]{\mathrm{Diag}}
\newcommand{\ff}{_{\mathrm{F}}}
\newcommand{\fs}{^2_{\mathrm{F}}}
\newcommand{\tp}{^\top}

\newcommand{\TX}{{\ca{T}_{\X}}}
\newcommand{\NX}{\ca{N}_{\X}}

\newcommand{\xk}{{x_{k} }}
\newcommand{\yk}{{y_{k} }}
\newcommand{\xkp}{{x_{k+1} }}

\newcommand{\A}{\ca{A}}

\newcommand{\Jc}{{\nabla c}}

\newcommand{\X}{{ \ca{X} }}
\newcommand{\K}{{ \ca{K} }}
\newcommand{\M}{{ \ca{M} }}
\newcommand{\Z}{{\ca{X}_{\tilde{\rho}}}}
\newcommand{\W}{{\ca{X}_1}}
\newcommand{\Y}{{\ca{K}_\rho}}

\newcommand{\Rn}{\mathbb{R}^n}
\newcommand{\Rp}{\mathbb{R}^p}
\newcommand{\Rm}{\mathbb{R}^m}

\newcommand{\IX}{\mathbb{I}_{\mathcal{X}}}

\newtheorem{theo}{Theorem}[section]
\newtheorem{lem}[theo]{Lemma}
\newtheorem{prop}[theo]{Proposition}

\newtheorem{defin}[theo]{Definition}

\newtheorem{assumpt}[theo]{Assumption}

\newcommand{\comm}[1]{{\color{red}#1}}

\usepackage[figuresright]{rotating}
\usepackage{pdflscape}

\usepackage{caption}
\usepackage{algorithm}
\usepackage{algpseudocode}
\usepackage{longtable}
\usepackage{appendix}
\usepackage{authblk}

\numberwithin{equation}{section}

\usepackage{array}

\def\psimu{\psi_\mu}
\def\Tmu{T_\mu}

\title{Partial Envelope for Optimization Problem with Nonconvex Constraints}
\author{Xiaoyin Hu, ~Xin Liu, ~Kim-Chuan Toh, ~Nachuan Xiao}

\begin{document}
\maketitle
	
\begin{abstract}
    In this paper, we consider the nonlinear constrained optimization problem (NCP) with constraint set $\{x \in \mathcal{X}: c(x) = 0\}$, where $\mathcal{X}$ is a closed convex subset of $\mathbb{R}^n$. Building upon the forward-backward envelope framework for optimization over $\mathcal{X}$, we propose a forward-backward semi-envelope (FBSE) approach for solving (NCP). In the proposed semi-envelope approach, we eliminate the constraint $x \in \mathcal{X}$ through a specifically designed envelope scheme while preserving the constraint $x \in \mathcal{M} := \{x \in \mathbb{R}^n: c(x) = 0\}$. We establish that the forward-backward semi-envelope for (NCP) is well-defined and locally Lipschitz smooth over a neighborhood of $\mathcal{M}$. Furthermore, we prove that (NCP) and its corresponding forward-backward semi-envelope have the same first-order stationary points within a neighborhood of $\mathcal{X} \cap \mathcal{M}$. Consequently, our proposed forward-backward semi-envelope approach enables direct application of optimization methods over $\mathcal{M}$ while inheriting their convergence properties for (NCP). Additionally, we develop an inexact projected gradient descent method for minimizing the forward-backward semi-envelope over $\mathcal{M}$ and establish its global convergence. Preliminary numerical experiments demonstrate the practical efficiency and potential of our proposed approach.
\end{abstract}

\section{Introduction}
In this paper, we consider the following constrained optimization problem,
\begin{equation}
    \tag{NCP}
    \label{Prob_Ori}
        \begin{aligned}
            \min_{x \in \Rn} \quad &f(x) + \IX(x)\\
            \text{s. t.} \quad & x \in \M := \{x \in \Rn: c(x) = 0\},
        \end{aligned}
\end{equation}
where $\X$ is a compact convex subset of $\Rn$ with easy-to-compute projection mapping. Moreover, $\IX(x)$ denotes the indicator function of the set $\X$, defined by $\IX(x) = 0$ for any $x \in \X$ and $\IX(x) = +\infty$ otherwise. Therefore,  \eqref{Prob_Ori} can be viewed as the optimization problem that minimizes $f(x)$ over $\{x \in \X: c(x) = 0\}$. Let $\mathrm{lin}(\ca{C})$ be the largest subspace of $\Rn$ that is contained in $\ca{C}$ for any closed convex cone $\mathcal{C}$ (i.e.,  $\mathrm{lin}(\ca{C}) = \ca{C} \cap -\ca{C}$), we make the following assumptions on \eqref{Prob_Ori}. 
    \begin{assumpt}
        \label{Assumption_f}
        \begin{enumerate}
            \item The objective function $f: \Rn \to \bb{R}$ is twice-differentiable over $\Rn$. 
            \item The constraint mapping $c: \Rn \to \bb{R}^p$ is twice-differentiable and has locally Lipschitz continuous second-order derivatives over $\Rn$.
            \item For any given $x \in \K := \X \cap \M$, it holds that 
            \begin{equation}
                 \nabla c(x)\tp \mathrm{lin}(\TX(x)) = \Rp.
            \end{equation} 
            That is, the constraint nondegeneracy condition \cite{robinson1980strongly} holds over $\X \cap \M$. 
        \end{enumerate}
    \end{assumpt}

    Optimization problems that take the form of \eqref{Prob_Ori} cover various important optimization models, including optimization with equality and inequality constraints \cite{jorge2006numerical}, conic programming \cite{tang2024feasible,liang2021inexact,hou2025low,xiao2025exact}, and manifold optimization \cite{zass2006nonnegative,absil2008optimization,hu2022constraint}. Consequently, problems of the form \eqref{Prob_Ori} have wide applications in various areas, such as learning tasks \cite{barlow1989unsupervised,xiao2025cdopt}, signal processing \cite{candy1986signal,luo2010semidefinite}, mechanical design \cite{venkayya1978structural}, etc. 

    The envelope approaches play a crucial role in analyzing and solving optimization problems with composite structures. Consider the following optimization problem with an indicator term,
    \begin{equation}
    \label{Prob_Simple}
    \min_{x \in \Rn} \quad h(x) + \IX(x),
    \end{equation}
    where $h$ is a Lipschitz continuous function over $\Rn$. To address the nonsmoothness and discontinuity introduced by the term $\IX(x)$, envelope approaches have been developed to transform \eqref{Prob_Simple} into an unconstrained minimization of the corresponding envelope function, which is continuously differentiable over $\Rn$.
    
    For example, inspired by the forward-backward splitting \cite{chen1997convergence,attouch2013convergence,raguet2013generalized}, the forward-backward envelope \cite{patrinos2013proximal,liu2017further,stella2017forward} is developed as a continuous differentiable surrogate for $h(x) + \IX(x)$. By assuming the continuous differentiability of $h$, the formulation of forward-backward envelope can be expressed as 
    \begin{equation}
        e_{F, \mu}(x) := \min_{w \in \X} h(x) + \inner{ \nabla h(x), w-x} + \frac{1}{2\mu}\norm{w-x}^2. 
    \end{equation}
    Here $\mu > 0$ refers to the envelope parameter. Then for any sufficiently small $\mu$, any $x \in \X$ satisfying $0 \in \nabla e_{F,\mu}(x)$ is also a first-order stationary point of \eqref{Prob_Simple}.
    More detailed properties of $e_{F, \mu}$ are discussed in \cite{stella2017forward,themelis2018forward,ahookhosh2021bregman}, including its Lipschitz continuity, differentiability, Lipschitz continuity of its gradients, and second-order properties. Furthermore, the forward-backward envelope enables the development of various efficient methods for solving \eqref{Prob_Simple}, including proximal gradient methods \cite{themelis2018forward}, quasi-Newton methods \cite{stella2017forward}, semi-smooth Newton methods \cite{wu2025globalized}, and penalty functions \cite{hu2024minimization}.
    
    Additionally, when $h$ is weakly convex and possibly nonsmooth, the Moreau envelope \cite{moreau1965proximite} can be employed as a smooth surrogate for  \eqref{Prob_Simple},
    \begin{equation}
        e_{M,\mu}(x) := \min_{w \in \X} h(w) + \frac{1}{2\mu} \norm{w-x}^2,
    \end{equation}
    where $\mu > 0$ refers to the envelope parameter. 
    For any sufficiently small $\mu$, \eqref{Prob_Simple} and $e_{M,\mu}$ share the same first-order stationary points over $\Rn$. However, the computation of Moreau envelope is typically intractable in practice since its corresponding subproblems usually do not admit closed-form solutions. Despite this, the Moreau envelope has been widely used in the theoretical analysis to minimize weakly convex function. Based on the concept of Moreau envelope, a wide range of optimization methods are analyzed, such as the stochastic subgradient method \cite{davis2019stochastic}, proximal point methods \cite{davis2022escaping}, Lagrangian-based methods \cite{zeng2022moreau}, bilevel optimization \cite{liu2024moreau}, minimax optimization \cite{grimmer2023landscape}, etc.

    However, all these existing envelope approaches for \eqref{Prob_Simple} rely on the convexity of $\X$.  Although \eqref{Prob_Ori} can be viewed as an unconstrained optimization problem with the indicator term $\bb{I}_{\M \cap \X}$, the nonconvexity of $\M \cap \X$ leads to the nonsmoothness of $e_{F, \mu}$ and $e_{M, \mu}$. Furthermore, the efficient computation of the gradients of the envelope functions for \eqref{Prob_Simple} relies on the efficient computation of the projection to $\X$. However, when treating \eqref{Prob_Ori} as an unconstrained optimization problem with the indicator term $\bb{I}_{\M \cap \X}$, the projection onto $\X \cap \M$ can be computationally expensive, even if both $\Pi_{\M}$ and $\Pi_{\X}$ are easy to compute. To the best of our knowledge, there is currently no existing work that directly develops envelope approaches for \eqref{Prob_Ori} by treating the constraint $x \in \M \cap \X$ as the indicator term $\bb{I}_{\M \cap \X}$.

    Very recently, \cite{xiao2025exact} proposes an exact penalty function for \eqref{Prob_Ori} that is formulated as follows,
    \begin{equation}
        \label{Prob_CDP}
        \min_{x \in \X} \quad  h_{\rm cdf}(x) := f(\A(x)) + \frac{\beta}{2}\norm{c(x)}^2. 
    \end{equation}
    Here $\A: \X \to \bb{R}$ is a specifically designed constraint dissolving mapping, which is locally Lipschitz smooth over $\X$. As demonstrated in \cite{xiao2025exact}, the construction of $\A$ is independent of the choices of $f$, and can be constructed using only $\Jc(x)$ and the {\it projective mapping} associated to $\X$. We call mapping $Q: \Rn \to \bb{S}^{n\times n}_+$ a projective mapping if it satisfies the following conditions.  
    \begin{assumpt}
        \label{Assumption_Q}
        \begin{enumerate}
            \item $Q: \Rn \to \bb{S}^{n\times n}_+$ is locally Lipschitz smooth over $\Rn$,  where $\bb{S}^{n\times n}_+$ is  the set of $n\times n$ symmetric positive semi-definite matrices.
            \item For any $x \in \X$, it holds that $\mathrm{null}(Q(x)) = \mathrm{range}(\NX(x))$. 
        \end{enumerate}
    \end{assumpt}
    Then based on the concept of the projective mapping, the constraint dissolving mapping $\A: \Rn \to \Rn$ for \eqref{Prob_Ori} can be expressed as,
    \begin{equation}
        \label{Eq_constraint_dissolving_mapping}
        \A(x) = x - Q(x)\Jc(x) (\Jc(x)\tp Q(x) \Jc(x) + \tau(x) I_p )^{-1}c(x),
    \end{equation} 
    where $\tau(x)$ is some scalar function that will be specified later.

    As shown in \cite{xiao2025exact}, the constraint dissolving approach decouples the constraints $c(x) = 0$ and $x \in \X$ through exact penalization. Although we can further smooth \eqref{Prob_CDP} by the Moreau envelope or forward-backward envelope, these approaches rely on the choices of the penalty parameter, which could be challenging to choose in practice. Therefore, we are led to ask the following question:
    \begin{quote}
        Can we develop envelope approaches for constrained optimization problems in the form of \eqref{Prob_Ori}, without introducing any penalty parameters?
    \end{quote}

    In this paper, we consider an alternative envelop approach for \eqref{Prob_Ori} without introducing any penalty parameter. Based on the concept of the projective mapping, we introduce the following forward-backward semi-envelope (FBSE) for \eqref{Prob_Ori},
    \begin{equation}
        \tag{FBSE}
        \label{Prob_FBE}
        \begin{aligned}
            \min_{x \in \Rn}\quad &\psimu(x)\\
            \text{s. t.} \quad & x \in \M,
        \end{aligned}
    \end{equation}
    where $\psimu(x) := \min_{w \in \X} ~ f(x) + \inner{J(x) \nabla f(x), w-x} + \frac{1}{2\mu}\norm{w-x}^2$. 
    Moreover,  $\mu > 0$ is the envelope parameter, and the mapping 
    $J: \Rn \mapsto \Rn$ is defined by,
    \begin{equation}
        J(x) := I_n - \Jc(x) (\Jc(x)\tp Q(x) \Jc(x) + \tau(x) I_p )^{-1}\Jc(x)\tp Q(x).
    \end{equation}
    Here the mapping $J(x)$ is an approximation to the Jacobian of $\A$ at $x$, while $\tau(x) := L_{\tau}(\norm{c(x)}^2 + \mathrm{dist}(x, \X)^2)$ with a prefixed parameter $L_{\tau} > 0$. Moreover, for any $x \in \Rn$, we denote $\Pi_{\X}(x)$ as the projection onto $\X$ and 
    \begin{equation}
        \label{Eq_Def_TF}
        \Tmu(x) := \mathop{\arg\min}_{w \in \X} ~ f(x) + \inner{J(x) \nabla f(x), w-x} + \frac{1}{2\mu}\norm{w-x}^2
        \;=\; \Pi_{\X}\big(x-\mu J(x) \nabla f(x)\big). 
    \end{equation}
    From the convexity of $\X$, for any $x \in \X$, $\Tmu(x)$ is a singleton and satisfies the following optimality condition,
    \begin{eqnarray}
    \label{Eq_optim_Tmu}
      0 \in \frac{1}{\mu}(\Tmu(x)-x)+ J(x)\nabla f(x) + {\ca N}_\X(\Tmu(x)).
    \end{eqnarray}
    
    
    We name the envelope $\psimu(\cdot)$ in \eqref{Prob_FBE} as the forward-backward semi-envelope for \eqref{Prob_Ori}, as it waives the nonsmooth term $\IX(x)$ from \eqref{Prob_Ori} while keeping the constraint $x \in \M$ unchanged. By assuming the local Lipschitz smoothness of $f$ in Assumption \ref{Assumption_f}(1), we prove that $\psimu(x)$ is continuously differentiable over $\Rn$. Moreover, there exists $\bar{\mu} > 0$ such that \eqref{Prob_Ori} and \eqref{Prob_FBE} share the same first-order stationary point over a neighborhood of $\X$ for any $\mu \in (0, \bar{\mu}]$. 

    We prove that there exists a neighborhood $\Y$ of $\X\cap \M$ such that the Jacobian $\nabla c(x)$ is full-rank over $\Y $, hence the constraints $c(x) =0$ satisfy the linear independence constraint qualification (LICQ) over $ \Y $.
    The equivalence between \eqref{Prob_Ori} and \eqref{Prob_FBE} then enables the direct implementation of various optimization approaches for solving the equality constrained optimization \eqref{Prob_FBE}, such as the sequential quadratic programming (SQP) method \cite{berahas2021sequential}, the augmented Lagrangian method \cite{xie2021complexity}, etc. Furthermore, we develop a projected inexact gradient descent method for solving \eqref{Prob_FBE}, where we employ $\frac{1}{\mu}(x - \Tmu(x))$ as an inexact evaluation of $\nabla \psimu$. Although the generic projected inexact gradient method typically lacks convergence guarantees, we prove that our proposed method has $\ca{O}(\varepsilon^{-2})$ iteration complexity. 

    Additionally, it is worth mentioning that we can consider the Moreau semi-envelope for \eqref{Prob_Ori} based on similar techniques as the forward-backward sem-envelope, which can be formulated as $\psi_{M, \mu}(x) := \mathop{\arg\min}_{y \in \X} ~f(y) + \frac{1}{2\mu}\norm{y-x}^2$. However, as the subproblem of the Moreau semi-envelope does not admit a closed-form solution and is generally intractable in practice, we omit the discussion of the Moreau semi-envelope $\psi_{M, \mu}(x)$ to  focus our paper on the more practical forward-backward semi-envelope $\psimu$ for \eqref{Prob_Ori}.

    The outline of the rest of this paper is as follows. In Section 2, we present the notations and preliminary concepts that are necessary for the proofs in this paper. In Section 3, we prove the equivalence between \eqref{Prob_Ori} and \eqref{Prob_FBE}, and present the inexact projected gradient method. Preliminary numerical experiments are given in Section 4 to demonstrate the efficiency of our proposed inexact projected gradient method. We conclude the paper in the last section.

\section{Preliminary}

\subsection{Notation}
For any matrix $A\in\bb{R}^{n\times p}$, 
let $\mathrm{range}(A)$ be the subspace spanned by the column vectors of $A$, $\mathrm{null}(A)$ be the null space of $A$ (i.e., $\mathrm{null}(A) = \{d \in \Rp: A d = 0\}$), 
and $\norm{\cdot}$ denotes the $\ell_2$-norm of a vector or an operator. 
For a subset $\ca{C} \subseteq \Rn$, $\mathrm{range}(\ca{C})$ refers to the smallest subspace of $\Rn$ that contains $\ca{C}$,  $\mathrm{aff}(\ca{C})$ refers to the affine hull of $\ca{C}$, and $\mathrm{ri}(\ca{C})$ refers to the relative interior of $\ca{C}$. Moreover, when $\ca{C}$ is a subspace of $\Rn$, $\ca{C}^{\perp}$ is defined as the largest subspace that is orthogonal to $\ca{C}$. Additionally, for any $w \in \Rn$, we use the notation $\inner{w,\ca{C}} := \{ \inner{w,d} : d\in \ca{C}\}$ and $w^{\perp} := \{w\}^{\perp}$. And we denote $\ca{B}_{\rho}(y)$ as $\left\{ x \in \mathbb{R}^n \mid \| y - x \| < \rho \right\}$.

The notations $\mathrm{diag}(A)$ and $\Diag(x)$
stand for the vector formed by the diagonal entries of a matrix $A$,
and the diagonal matrix with the entries of $x\in\bb{R}^n$ as its diagonal, respectively. 
We denote the $r$-th largest singular value of a matrix $A\in \bb{R}^{n\times p}$ by $\sigma_r(A)$, while $\sigma_{\min}(A)$ refers to the smallest singular value of 
$A$. Furthermore, 
the pseudo-inverse of $A$ is denoted by $A^\dagger \in \bb{R}^{p\times n}$, which satisfies $AA^\dagger A = A$, $A^\dagger AA^\dagger = A^\dagger$, and both $A^{\dagger} A$ and $A A^{\dagger}$ are symmetric.

For any closed subset $\ca{C} \subseteq \Rn$ and any $x \in \bb{R}^n$, $\mathrm{range}(\ca{C})$ refers to the smallest subspace of $\Rn$ that contains $\ca{C}$. Moreover, we define the projection from $x \in \bb{R}^n$ to $\ca{C}$ as 
\begin{equation*}
	\Pi_{\ca{C}}(x) := \mathop{\arg\min}_{y \in \ca{C}} ~ \norm{x-y}. 
\end{equation*}  
Furthermore, $\mathrm{dist}(x, \ca{C})$ refers to the distance between $x$ and $\ca{C}$, i.e. $ \mathrm{dist}(x, \ca{C}) = \norm{x - \Pi_{\ca{C}}(x)}$.

For any $m \geq 1$ and any differentiable mapping $T: \Rn \to \Rm$, its (transposed) Jacobian is denoted as  $\nabla T(x) \in \bb{R}^{n\times m}$. More precisely, let $T_i$ be the $i$-th coordinate of the mapping $T$, then $\nabla T(x)$ can be expressed as 
\begin{equation*}
    \nabla T(x) = \Big[\nabla T_1(x), \ldots, \nabla T_p(x)\Big] \in \bb{R}^{n\times m}. 
\end{equation*}
In particular, $\nabla c(x) \in \bb{R}^{n\times p}$, and for any
$x\in\M$, the tangent space of $\M$ at $x$ is given by 
${\rm null}(\nabla c(x)^\top)$, and the normal space of $\M$ at $x$ can be expressed as $\mathrm{range}(\nabla c(x))$.  


\subsection{Stationarity}
\label{Subsection_Stationarity}
In this subsection, we introduce the definitions of the stationary points for the constrained optimization problem \eqref{Prob_Ori} and the optimization problem over $\X \cap \M$. We begin with the first-order optimality condition of \eqref{Prob_Ori}, which is defined as follows.
\begin{defin}[\cite{clarke1990optimization}]\label{Defin_FOSP}
	Given $x \in \X \cap \M$, we say that $x$ is a first-order stationary point of \eqref{Prob_Ori} if
	\begin{equation*}
            0 \in \nabla f(x) +  \mathrm{range}(\Jc(x)) + \NX(x).
	\end{equation*}
    Moreover, for any given $\varepsilon > 0$, we say that $x \in \X \cap \M$ is an $\varepsilon$-first-order stationary point of \eqref{Prob_Ori} if 
    \begin{equation*}
        \mathrm{dist}\left( 0, \nabla f(x) + \mathrm{range}(\Jc(x)) + \NX(x) \right) \leq \varepsilon, \quad \text{and} \quad \norm{c(x)} \leq \varepsilon. 
    \end{equation*}
\end{defin}
Next, we give the definitions of the first-order stationary point and the $\varepsilon$-first-order stationary point of the following optimization problem
\begin{equation}
    \label{Prob_EqCon}
    \begin{aligned}
        \min_{x \in \Rn} \quad &h(x)\\
        \text{s. t.} \quad & x \in \M = \{x \in \Rn: c(x) = 0\}.
    \end{aligned}
\end{equation}

\begin{defin}
    Given $x \in \M$, we say that $x$ is a first-order stationary point of \eqref{Prob_EqCon} if
    \begin{equation*}
        0\in \nabla h(x) + \mathrm{range}(\Jc(x)) .
    \end{equation*}
    Moreover, for any given $\varepsilon > 0$, we say that $x \in \M$ is an $\varepsilon$-first-order stationary point of \eqref{Prob_EqCon} if
    \begin{equation*}
        \mathrm{dist}\left( 0, \nabla h(x) + \mathrm{range}(\Jc(x))  \right) \leq \varepsilon. 
    \end{equation*}
\end{defin}

\section{Forward-backward Semi-Envelope for \eqref{Prob_Ori}}
In this section, we aim to analyze the forward-backward envelop for \eqref{Prob_Ori}. This section is structured as follows. Section \ref{Subsection_constants} introduces the basic concepts and constants that are essential for the theoretical analysis. Section \ref{Subsection_basic_property} establishes the equivalence between \eqref{Prob_Ori} and \eqref{Prob_FBE} in a neighborhood of $\X$. Finally, Section \ref{Subsection_algorithm} presents a globally convergent projected inexact gradient method and its convergence properties.

\subsection{Basic constants}
\label{Subsection_constants}
In this subsection, we first introduce some necessary constants in our theoretical analysis under Assumption \ref{Assumption_f} and Assumption \ref{Assumption_Q} in Table \ref{Simbol_Definition}. 
\begin{table}[h!]
\centering
\begin{tabular}{c p{0.75\textwidth}}
\toprule
\textbf{Notation} & \textbf{Definition} \\
\midrule
$\sigma_c$ & $\inf_{x \in \X\cap \M}\sigma_{\min}(\Jc(x))$ 
\\[2pt]
$\sigma_{Q}$ & $\inf_{x \in \X\cap \M}\sigma_{\min}(\Jc(x)\tp Q(x) \Jc(x))$ 
\\[2pt]
$\W$ &  $\{x \in \Rn: \mathrm{dist}(x, \X) \leq 1\}$ 
\\[2pt]
$M_{f}$ & $\sup_{x \in \W} \norm{\nabla f(x)}$ 
\\[2pt]
$L_{f}$ & $\sup_{x, y\in \W, x \neq y} \frac{\norm{\nabla f(x) - \nabla f(y)}}{\norm{x-y}}$ 
\\[5pt]
$M_{Q}$ & $\sup_{x \in \W} \norm{Q(x)}$ 
\\[5pt]
$L_{Q}$ & $\sup_{x, y \in \W, ~ x\neq y} \frac{\norm{Q(x) - Q(y)}}{\norm{x-y}}$ 
\\[5pt]
$L_{D}$ & $\sup_{x, y \in \W, ~ x\neq y} \frac{\norm{\Jc(x)\tp Q(x) \Jc(x) - \Jc(y)\tp Q(y) \Jc(y)}}{\norm{x-y}}$ 
\\[5pt]
$M_c$ & $\sup_{x \in \W} \norm{\nabla c(x)}$
\\[5pt]
$L_c$ & $\sup_{x, y\in \W, x \neq y} \frac{\norm{\nabla c(x) - \nabla c(y)}}{\norm{x-y}}$
\\[5pt]
$\sigma_{\mathrm{res}}$ & $\inf \{\norm{c(x)}: x \in \W, ~\mathrm{dist}(x, \M) \geq \frac{\sigma_Q}{2L_D} \}$
\\[2pt]
$\tilde{\rho}$ & $\min\left\{1, \sigma_{\mathrm{res}}^2\right\}$
\\[8pt]
$\Z$ & $\{x \in \Rn: \mathrm{dist}(x, \X) \leq \tilde{\rho}\}$ \\
\bottomrule
\end{tabular}
\caption{The definition of some necessary constants in theoretical analysis.}
    \label{Simbol_Definition}
\end{table}

It is worth mentioning that Assumption \ref{Assumption_f}(3) guarantees that $\sigma_c > 0$ and $\sigma_{Q} > 0$. Additionally, it follows from the compactness of $\X$ that $M_Q$, $L_Q$ and $L_D$ are finite.

Next, we present the following lemma illustrating the well-posedness of $J(x)$ and $\A(x)$. 
\begin{lem}
    \label{Le_aux_1}
    Suppose Assumption \ref{Assumption_f} and Assumption \ref{Assumption_Q} hold. Then for any $x \in \Rn$, it holds that 
    \begin{equation}
        \Jc(x)\tp Q(x) \Jc(x) + \tau(x) I_p \succeq \min\left\{\frac{\sigma_{Q}}{2},  L_{\tau} \tilde{\rho}\right\} I_p. 
    \end{equation}
\end{lem}
\begin{proof}
    As demonstrated in  Assumption \ref{Assumption_Q}, $Q(x) \succeq 0$ holds for any $x \in \Rn$. As a result, it holds that $\Jc(x)\tp Q(x) \Jc(x) \succeq 0$ holds for any $x \in \Rn$. Notice that for any $x \in \X \cap \M$, the constraint nondegeneracy condition in Assumption \ref{Assumption_f}(3) guarantees that $\Jc(x)\tp Q(x) \Jc(x) \succ 0$.  Furthermore, from the continuity of $Q(x)$ and $\Jc(x)$ in Assumption \ref{Assumption_f} and Assumption \ref{Assumption_Q}, it holds that the mapping $x \mapsto \Jc(x)\tp Q(x) \Jc(x)$ is continuous over $\Rn$. As a result, from the definition of $L_D$ and that $\tilde{\rho}\leq 1$, 
    it holds that for any 
    $x \in \{x \in \Z: 
    \mathrm{dist}(x, \X \cap \M) \leq \frac{\sigma_Q}{2L_D}\}$, we have for any $y \in \X \cap \M$,
    \begin{equation}
       \label{eq-3.2}
        \sigma_{\min}(\Jc(x)\tp Q(x) \Jc(x)) \geq \sigma_{\min}(\Jc(y)\tp Q(y) \Jc(y)) - L_D \norm{y-x} \geq \sigma_{Q} - L_D \cdot \frac{\sigma_Q}{2L_D} \geq \frac{\sigma_{Q}}{2}. 
    \end{equation}
    As a result, we can conclude that $\Jc(x)\tp Q(x) \Jc(x) + L_{\tau} \norm{c(x)}^2 I_p \succ 0$ holds for any $x \in \{x \in \Z: \mathrm{dist}(x, \X \cap \M) \leq \frac{\sigma_Q}{2L_D}\}$.

    Furthermore, for any $x \in \{x \in \Z: 
    \mathrm{dist}(x, \X \cap \M) > \frac{\sigma_Q}{2L_D}\}$, 
    from the fact that $\tilde{\rho} \leq \sigma_{\mathrm{res}}^2$, we have 
    \begin{equation*}
        \begin{aligned}
            \sigma_{\min}(\Jc(x)\tp Q(x) \Jc(x) + \tau (x) I_p) \geq \sigma_{\min}(\Jc(x)\tp Q(x) \Jc(x)) + L_\tau \sigma_{\mathrm{res}}^2 \geq L_\tau \sigma_{\mathrm{res}}^2 \geq L_{\tau} \tilde{\rho}.
        \end{aligned}
    \end{equation*}
    Additionally, for any $x \notin \Z$, we have that $\mathrm{dist}(x, \X) > \tilde{\rho}$. Therefore, it directly follows from the choice of $\tau$ that 
    \begin{equation*}
         \Jc(x)\tp Q(x) \Jc(x) + \tau(x) I_p \succeq  \tau(x) I_p \succeq  L_{\tau} \tilde{\rho} I_p. 
    \end{equation*}
    As a result, we can conclude that $\Jc(x)\tp Q(x) \Jc(x) + \tau(x) I_p \succeq \min\left\{\frac{\sigma_{Q}}{2}, L_{\tau} \tilde{\rho} \right\} I_p \succ 0$ holds for any $x \in \Rn$. This completes the proof. 
\end{proof}

Lemma \ref{Le_aux_1} illustrates the well-definedness of $J(x)$ over $\Rn$. Moreover, from the Lipschitz smoothness of $\Jc(x)$ over $\Rn$, we can conclude the Lipschitz smoothness of $J(x)$ over $\W$. Then we make the following definitions on the notations in Table \ref{Simbol_Definition_2}. 
\begin{table}[h!]
\centering
\small 
\begin{tabular}{c p{0.85\textwidth}}
\toprule
\textbf{Notation} & \textbf{Definition} \\
\midrule
$M_{J}$ & $\sup_{x \in \W} \norm{J(x)}$ \\[5pt]
$L_{J}$ & $\sup_{x, y \in \W, ~ x\neq y} \frac{\norm{J(x) - J(y)}}{\norm{x-y}}$ \\[5pt]
$M_{H}$ & $\sup_{x \in \W} \norm{H(x)}$ \\[5pt]
$L_{H}$ & $\sup_{x, y \in \W, ~ x\neq y} \frac{\norm{H(x) - H(y)}}{\norm{x-y}}$ \\[5pt]
$\rho$ & $\min\left\{\tilde{\rho}, 
\frac{\sigma_Q L_Q M_J}{M_QL_\tau}, 
\frac{\sigma_Q}{8M_c^2 L_Q},\, \frac{\sigma_c}{8L_c}\right\}$ 
\\[5pt]
$\Y$ & $\{x \in \M: \mathrm{dist}(x, \X) \leq \rho\}. $
\\[5pt]
$\mu_{\max}$ & $\min\left\{ 
\begin{array}{l}
\frac{\rho}{64(M_J M_f+1)}, \, 
\frac{\sigma_{Q}}{16  M_c^2L_QM_JM_f  + (16M_c^2M_Q+4\sigma_Q) M_H },\, \frac{1}{4L_H+12M_H+4L_f},
\\[5pt]
\frac{\rho^2}{16( \sup_{w, z \in \W} f(w) - f(z))},\, 
\frac{\sigma_c}{  6(M_J+1) M_f L_c}
\end{array}
\right\}$.\\
\bottomrule
\end{tabular}

\caption{The definition of some necessary constants in the theoretical analysis based on Lemma \ref{Le_aux_1}. Here the mapping $H(x) := \nabla J(x)[\nabla f(x)] + J(x) \nabla^2 f(x)$.}
    \label{Simbol_Definition_2}
\end{table}

\subsection{Basic properties of forward-backward semi-envelope}
\label{Subsection_basic_property}

In this subsection, we aim to establish the equivalence between \eqref{Prob_Ori} and \eqref{Prob_FBE}. We begin our theoretical analysis with several technical lemmas. The following lemma illustrates the properties of $J(x) \nabla c(x)$. 
\begin{lem}
    \label{Le_aux_2}
    Suppose Assumption \ref{Assumption_f} and Assumption \ref{Assumption_Q} hold. Then for any $x \in \Rn$, it holds that 
    \begin{equation}
        J(x) \nabla c(x) = \tau(x)\Jc(x)(\Jc(x)\tp Q(x) \Jc(x) + \tau(x) I_p )^{-1}.
    \end{equation}
    In particular, $J(x)\Jc(x) = 0$ holds for any $x \in \X \cap \M$.
\end{lem}
\begin{proof}
    As demonstrated in Lemma \ref{Le_aux_1}, $J(x)$ is well-defined over $\Rn$. Then it directly follows from the definition of $J(x)$ that, for any $x \in \Rn$, we have
    \begin{equation*}
        \begin{aligned}
            &J(x) \nabla c(x) = \nabla c(x) - \Jc(x) (\Jc(x)\tp Q(x) \Jc(x) + \tau(x) I_p )^{-1} \big(\Jc(x)\tp Q(x)\nabla c(x)\big)\\
            ={}& \nabla c(x) - \nabla c(x) (\Jc(x)\tp Q(x) \Jc(x) + \tau(x) I_p )^{-1} \big(\Jc(x)\tp Q(x) \Jc(x) + \tau(x) I_p - \tau(x)I_p\big)\\
            ={}& \tau(x) \nabla c(x)  (\Jc(x)\tp Q(x) \Jc(x) + \tau(x) I_p )^{-1}.
        \end{aligned}
    \end{equation*}
    This completes the proof. 
\end{proof}

Moreover, we present the following lemma showing that $J(x)$ is an identity mapping when restricted to the subspace $\mathrm{range}(\NX(x))$. The proof of Lemma \ref{Le_aux_3} directly follows from the Assumption \ref{Assumption_Q}(2), hence is omitted for simplicity.  
\begin{lem}
    \label{Le_aux_3}
    Suppose Assumption \ref{Assumption_f} and Assumption \ref{Assumption_Q} hold. Then for any $x \in \Rn $ and any $d \in \mathrm{range}(\NX(x))$, it holds that $J(x) d = d$. 
\end{lem}

Next we present the following lemma that establishes the existence of $\psimu(x)$ for $x\in\Rn$. 
\begin{lem}
    \label{Le_FBE_envelope_C1}
    Suppose Assumption \ref{Assumption_f} and Assumption \ref{Assumption_Q} hold, then $\psimu(x)$ exists and is locally Lipschitz continuous over $\Rn$. 
\end{lem}
\begin{proof}
    As demonstrated in Lemma \ref{Le_aux_1}, $J(x)$ is well-defined over $\Rn$. Therefore, from the formulation of $\psimu$, it holds that $\psimu(x)$ exists for any $x \in \Rn$. Moreover, from the locally Lipschitz continuity of $J(x)$ and $\nabla f(x)$ over $\Rn$, we can conclude that $\psimu(x)$ is locally Lipschitz continuous over $\Rn$. This completes the proof. 
\end{proof}

Then the following lemma establish the relationship between $\psimu(x)$ and $f(x)$. 
\begin{lem}
    \label{Le_FBE_fval_compare}
    Suppose Assumption \ref{Assumption_f} and Assumption \ref{Assumption_Q} hold, then for any $x \in \X$, it holds that 
    \begin{equation}
        \psimu(x) \leq f(x) - \frac{1}{2\mu} \norm{\Tmu(x) - x}^2.
    \end{equation}
\end{lem}
\begin{proof}
    From the optimality condition \eqref{Eq_optim_Tmu} of the proximal subproblem, it holds that 
    \begin{equation*}
        0 \in \frac{1}{\mu}(  \Tmu(x) - x ) + J(x) \nabla f(x) + \NX(\Tmu(x)).
    \end{equation*}
    The convexity of $\X$ guarantees that $ x - \Tmu(x)  \in \TX(\Tmu(x))$. Then it holds that $\inner{\Tmu(x) - x, d} \geq 0$ for any $d \in \NX(\Tmu(x))$, which leads to the following inequality,
    \begin{equation*}
        0 \geq \inner{\Tmu(x) - x, \frac{1}{\mu}(  \Tmu(x) - x ) + J(x) \nabla f(x)} = \frac{1}{\mu} \norm{\Tmu(x) - x}^2 + \inner{\Tmu(x) - x, J(x)\nabla f(x)}.
    \end{equation*}
    As a result, we can conclude that 
    \begin{equation*}
        \psimu(x) = f(x) + \inner{J(x)\nabla f(x), \Tmu(x) - x} + \frac{1}{2\mu}\norm{\Tmu(x) - x}^2 \leq  f(x) -\frac{1}{2\mu}\norm{\Tmu(x) - x}^2. 
    \end{equation*}
    This completes the proof. 
\end{proof}

Next, in the following lemma, we derive an upper bound for $\norm{\Tmu(x) - x}$. 
\begin{lem}
    \label{Le_FBE_Tx_x}
     Suppose Assumption \ref{Assumption_f} and Assumption \ref{Assumption_Q} hold. Then for any $\mu > 0$, it holds that 
    \begin{equation}
        \norm{\Tmu(x) - x} \leq M_{J} M_{f} \mu + \mathrm{dist}(x, \X),\quad \forall ~x \in \W.
    \end{equation}
\end{lem}
\begin{proof}
    For any $x \in \W$ and any $\mu > 0$, it follows from the expression of $\Tmu$ and the convexity of $\X$ that  
    \begin{equation*}
        \begin{aligned}
            &\norm{\Tmu(x) - x} = \norm{\Pi_{\X}(x - \mu J(x)\nabla f(x)) -  \Pi_{\X}(x) + \Pi_{\X}(x) - x} \\
            \leq{}& \norm{\mu J(x) \nabla f(x)} + \norm{\Pi_{\X}(x) - x} \leq M_{J} M_{f} \mu + \mathrm{dist}(x, \X). 
        \end{aligned}
    \end{equation*}
    This completes the proof. 
\end{proof}

Then we have the following lemma illustrating the expression of the gradient of $\psimu$.  
\begin{lem}
    \label{Le_FBE_gradient}
    Suppose Assumption \ref{Assumption_f} and Assumption \ref{Assumption_Q} hold. Then for any $x \in \Rn$, we have
    \begin{equation}
        \nabla \psimu(x) = \frac{1}{\mu}(I_n - \mu H(x))(x - \Tmu(x)) + (I_n-J(x))\nabla f(x),
    \end{equation}
    where $H(x) = J(x)\nabla^2 f(x) + \nabla J(x)[\nabla f(x)].$
    Moreover, for any $\mu \leq \mu_{\max}$, we have
    \begin{enumerate}
        \item $ \psimu(x)$ is $(3(M_J+1) M_f + \frac{2\rho}{\mu})$-Lipschitz continuous over $\Y$, and
        \item $\nabla \psimu$ is $\frac{2}{\mu}$-Lipschitz continuous over $\Y$.
    \end{enumerate}
    
\end{lem}

\begin{proof}
    The derivation of the expression of $\nabla \psimu(x)$  directly follows from \cite{rockafellar2009variational}, hence is omitted for simplicity.

    Next we consider any $x \in \Y$, from Lemma \ref{Le_FBE_Tx_x}, we can conclude that  
\begin{equation}
    \label{Eq_Ub_nabla_psi}
    \begin{aligned}
        &\norm{\nabla \psimu(x)} \leq \frac{1}{\mu}\norm{I_n - \mu H(x)} \norm{x - \Tmu(x)} + \norm{I_n - J(x)} \norm{\nabla f(x)} \\
        \leq{}& \frac{2(M_J M_f \mu + \rho)}{\mu} + (1+ M_J)M_f \leq 3(M_J+1) M_f + \frac{2\rho}{\mu}. 
    \end{aligned}
\end{equation}
Therefore,  $ \psimu(x)$ is $(3(M_J+1) M_f + \frac{2\rho}{\mu})$-Lipschitz continuous over $\Y$. 

Furthermore, as demonstrated in \cite{rockafellar2009variational}, the projection mapping $\Pi_{\X}$ is $1$-Lipschitz continuous over $\Rn$. Therefore, the Lipschitz constant for the mapping $x \mapsto I_n - \mu H(x)$ is $(1 + \mu L_{H})$, the Lipschitz constant for the mapping $x \mapsto \frac{1}{\mu} (x - \Tmu(x))$ is $\frac{1}{\mu}(1 + \mu M_H)$, and the Lipschitz constant for the mapping $x \mapsto (I_n - J(x))\nabla f(x)$ is $L_f + L_H$. As a result, the Lipschitz constant for $\nabla \psimu(x)$ over $\X$ is
\begin{equation}
(\mu M_J M_f+\rho)(1 + \mu L_{H}) + \frac{1}{\mu}(1 + \mu M_H)^2 + L_f + L_H,
\end{equation}
which is smaller than $\frac{2}{\mu}$ based on the definition of $\mu_{\max}$ in Table \ref{Simbol_Definition_2}. This completes the proof.

\end{proof}

Furthermore, we present the following lemma illustrating that $\Tmu(x)-x$ is nearly perpendicular to the subspace spanned by $Q(\Tmu(x))\Jc(x)$. 
\begin{lem}
    \label{Le_FBE_error_bound_cQc}
    Suppose Assumption \ref{Assumption_f} and Assumption \ref{Assumption_Q} hold. Then for any $x \in \Y$ and any $\mu > 0$, it holds that 
    \begin{equation}
        \norm{\frac{1}{\mu}(\Tmu(x)-x)\tp Q(\Tmu(x))\Jc(x)} \leq 2M_{c} L_{Q} M_{J} M_{f} \norm{\Tmu(x) - x}. 
    \end{equation}
\end{lem}
\begin{proof}
    For any $x \in \Y$ and any $\mu > 0$, from the optimality condition of $\Tmu(x)$, it holds that 
    \begin{equation*}
        0 \in \frac{1}{\mu}(\Tmu(x)-x) + J(x)\nabla f(x) + \NX(\Tmu(x)). 
    \end{equation*}
    Moreover, for any $x \in \Rn$, it holds that 
    \begin{equation*}
        \begin{aligned}
            &\nabla c(x)\tp Q(x) J(x) \\
            ={}& \nabla c(x)\tp Q(x)  - \nabla c(x)\tp Q(x) \nabla c(x)(\Jc(x)\tp Q(x) \Jc(x) + \tau(x) I_p )^{-1} \nabla c(x)\tp Q(x)\\
            ={}& \left(I_p - \nabla c(x)\tp Q(x) \nabla c(x)\left(\Jc(x)\tp Q(x) \Jc(x) + \tau(x) I_p \right)^{-1} \right)\nabla c(x)\tp Q(x)\\
            ={}& \tau(x) (\Jc(x)\tp Q(x) \Jc(x) + \tau(x) I_p )^{-1} \nabla c(x)\tp Q(x). 
        \end{aligned}
    \end{equation*}
    Together with the fact that $Q(\Tmu(x))\NX(\Tmu(x)) = \{0\}$ from Assumption \ref{Assumption_Q}(2), we have that
    \begin{equation*}
        \begin{aligned}
            &\frac{1}{\mu}\nabla c(x)\tp Q(\Tmu(x)) (\Tmu(x)-x) = -\nabla c(x)\tp Q(\Tmu(x)) J(x) \nabla f(x) \\
            ={}& -\nabla c(x)\tp (Q(\Tmu(x)) - Q(x)) J(x) \nabla f(x) -  \nabla c(x)\tp Q(x) J(x)\nabla f(x)\\
            ={}& -\nabla c(x)\tp (Q(\Tmu(x)) - Q(x)) J(x) \nabla f(x) \\
            & \comm{- } \tau(x) (\nabla c(x)\tp Q(x) \nabla c(x) + \tau(x) I_p)^{-1} \nabla c(x)\tp Q(x) \nabla f(x).
        \end{aligned}
    \end{equation*}
    Therefore, it holds that
    \begin{equation*}
        \begin{aligned}
            &\norm{\frac{1}{\mu}\nabla c(x)\tp Q(\Tmu(x)) (\Tmu(x)-x) }\\
            \leq{}& \norm{\nabla c(x)} \norm{Q(\Tmu(x)) - Q(x)} \norm{J(x)} \norm{\nabla f(x)} + \frac{L_{\tau}}{\sigma_Q} M_c M_Q M_f \mathrm{dist}(x, \X)^2 \\
            \leq{}& \norm{\nabla c(x)} \norm{Q(\Tmu(x)) - Q(x)} \norm{J(x)} \norm{\nabla f(x)}  + \frac{L_\tau}{\sigma_Q} M_c M_Q M_f \norm{\Tmu(x) - x}^2 \\
            \leq{}& 2M_{c} L_{Q}  M_{J}  M_{f}\norm{\Tmu(x) - x}. 
        \end{aligned}
    \end{equation*}
    Note that to derive the last inequality, we used the fact that $\norm{\Tmu(x)-x}\leq \rho\leq \frac{\sigma_Q L_Q M_J}{M_QL_\tau}$. This completes the proof. 
\end{proof}

Now we present the following proposition illustrating the equivalence between \eqref{Prob_FBE} and \eqref{Prob_Ori} for any feasible point. 
\begin{prop}
    \label{Prop_equivalence_FBE}
    Suppose Assumption \ref{Assumption_f} and Assumption \ref{Assumption_Q} hold. If $x \in \X \cap \M$  is a first-order stationary point of \eqref{Prob_Ori}, then $x$ is a first-order stationary point of \eqref{Prob_FBE} with any $\mu > 0$. 
\end{prop}
\begin{proof}
    For any $x \in \X \cap \M$ that is a first-order stationary point of \eqref{Prob_Ori}, it holds that 
    \begin{equation}
        \label{Eq_Prop_equivalence_FBE_0}
        0 \in \nabla f(x) + \mathrm{range}(\nabla c(x)) + \NX(x).
    \end{equation}
    Together with Lemma \ref{Le_aux_2} and Lemma \ref{Le_aux_3}, it holds that $J(x)\nabla c(x) = 0$ for any $x \in \X \cap \M$, and $J(x)d = d$ holds for any $d \in \NX(x)$. Therefore, by multiplying $J(x)$ to both sides of \eqref{Eq_Prop_equivalence_FBE_0}, it holds that
    \begin{equation*}
        0 \in J(x)\nabla f(x) + \NX(x).
    \end{equation*}
    Then from the optimality condition of the proximal subproblem in \eqref{Prob_FBE}, for any $\mu > 0$, it holds that 
    \begin{equation*}
        x \in \arg\min_{w \in \X} f(x) + \inner{J(x)\nabla f(x), w-x} +\frac{1}{2\mu} \norm{w-x}^2. 
    \end{equation*}
    Thus by definition, $\Tmu(x) = x$, hence 
    \begin{equation*}
        \nabla \psimu(x) = (I_n - J(x))\nabla f(x). 
    \end{equation*}
    Moreover, since 
    \begin{equation*}
        J(x) - I_n = -\Jc(x) \Big((\Jc(x)\tp Q(x) \Jc(x) + \tau(x) I_p )^{-1}\Jc(x)\tp Q(x) \Big),
    \end{equation*}
    it holds that $\mathrm{range}(J(x) - I_n) = \mathrm{range}(\nabla c(x))$. As a result, we can conclude that
    \begin{equation*}
        0 \in \nabla \psimu(x) + \mathrm{range}(\nabla c(x)).
    \end{equation*}
    Hence $x$ is a first-order stationary point of \eqref{Prob_FBE}. This completes the proof. 
\end{proof}

In the following theorem, we prove that in a neighborhood of $\X \cap \M$, any first-order stationary point of \eqref{Prob_FBE} is a first-order stationary point of \eqref{Prob_Ori}. 
\begin{theo}
    \label{Theo_Equivalence_FBE}
    Suppose Assumption \ref{Assumption_f} and Assumption \ref{Assumption_Q} hold. Then for any $\mu \leq \mu_{\max}$, if  $x \in \Y $ is a first-order stationary point of \eqref{Prob_FBE}, then $x$ is a first-order stationary point of \eqref{Prob_Ori}. 
\end{theo}
\begin{proof}
    For any $x \in \Y $ that is a first-order stationary point of \eqref{Prob_FBE}, we first prove that $\norm{\Tmu(x)-x}=0$. From the optimality condition of \eqref{Prob_FBE}, there exists $\lambda \in \Rp$ such that 
    \begin{equation}
        \label{Eq_Theo_Equivalence_FBE_0}
        0 = \frac{1}{\mu}(I_n - \mu H(x))(x - \Tmu(x)) + \Jc(x)\lambda. 
    \end{equation}
    Therefore, we can conclude that 
    \begin{equation*}
        \begin{aligned}
            &\frac{1}{\mu} \norm{\Tmu(x)-x} - \norm{\Jc(x) \lambda} \leq  \norm{\frac{1}{\mu}(\Tmu(x)-x) + \Jc(x)\lambda } \\
            ={}& \norm{H(x)(\Tmu(x) - x)} \leq  M_H \norm{\Tmu(x)-x}. 
        \end{aligned}
    \end{equation*}
    Hence
    \begin{equation}
        \label{Eq_Theo_Equivalence_FBE_1}
        \norm{\lambda} \geq \frac{1}{M_c}\norm{\Jc(x) \lambda} \geq \frac{1}{M_c}\left(\frac{1}{\mu} - M_{H}\right) \norm{\Tmu(x)-x}.
    \end{equation}

    On the other hand, by multiplying $\nabla c(x)\tp Q(\Tmu(x)) $ to both sides of \eqref{Eq_Theo_Equivalence_FBE_0}, it follows from Lemma \ref{Le_FBE_error_bound_cQc} that 
    \begin{equation}
        \label{Eq_Theo_Equivalence_FBE_2}
        \begin{aligned}
            &\norm{\nabla c(x)\tp Q(\Tmu(x)) \nabla c(x) \lambda} = \norm{\nabla c(x)\tp Q(\Tmu(x)) \left( \frac{1}{\mu}(I_n - \mu H(x))(x - \Tmu(x)) \right)} \\
            \leq{}& \norm{\frac{1}{\mu}\nabla c(x)\tp Q(\Tmu(x)) (\Tmu(x)-x) } +  M_{c} M_{Q} M_{H} \norm{\Tmu(x)-x} \\
            \leq{}& (2M_{c} L_{Q}  M_{f} M_{J}+ M_{c} M_{Q}  M_{H}) \norm{\Tmu(x)-x}.
        \end{aligned}
    \end{equation}
    
    From the proof of \eqref{eq-3.2}, we have that 
    $\sigma_{\min}(\nabla c(x)\tp Q(x) \nabla c(x)) \geq\frac{\sigma_{Q}}{2}$ for 
    $x\in \Y$. Then together with Lemma \ref{Le_FBE_Tx_x}, we have 
    \begin{equation}
        \label{Eq_Theo_Equivalence_FBE_3}
        \begin{aligned}
            &\sigma_{\min}(\nabla c(x)\tp Q(\Tmu(x)) \nabla c(x)) \geq \sigma_{\min}(\nabla c(x)\tp Q(x) \nabla c(x)) - L_Q M_c^2 \norm{\Tmu(x) - x}\\
            \geq{}& \frac{\sigma_{Q}}{2} -   L_Q M_c^2(M_J M_f \mu + \rho) \geq \frac{\sigma_{Q}}{4}. 
        \end{aligned}
    \end{equation}
    In deriving the last inequality, we used the fact that 
    $\mu\leq \frac{\rho}{64M_JM_f}$ and $\rho\leq \frac{\sigma_Q}{8  L_Q M_c^2}$.
    Substituting the inequality \eqref{Eq_Theo_Equivalence_FBE_3} into \eqref{Eq_Theo_Equivalence_FBE_2}, it holds that 
    \begin{equation*}
        \frac{\sigma_{Q}}{4} \norm{\lambda} \leq \norm{\nabla c(x)\tp Q(\Tmu(x)) \nabla c(x) \lambda} \leq (2M_{c} L_{Q} M_{J} M_{f} + M_{c} M_{Q}  M_{H}) \norm{\Tmu(x)-x}.
    \end{equation*}
    As a result, we can conclude that 
    \begin{equation*}
        \frac{\sigma_{Q}}{4M_c}\left(\frac{1}{\mu} - M_{H}\right) \norm{\Tmu(x)-x} \leq \frac{\sigma_{Q}}{4}  \norm{\lambda} \leq (2M_{c} L_{Q} M_{J} M_{f} + M_{c} M_{Q}  M_{H}) \norm{\Tmu(x)-x}. 
    \end{equation*}
    Notice that the choice of $\mu_{\max}$ guarantees that $\mu \leq \frac{1}{4M_H}$ and $\mu \leq \frac{\sigma_Q}{16  M_c^2L_Q M_J M_f  + 8M_c^2M_Q M_H }$. Thus
    \begin{equation*}
        \frac{\sigma_{Q}}{4M_c}\left(\frac{1}{\mu} - M_{H}\right) >\frac{\sigma_Q}{8M_c \mu} \geq (2M_{c} L_{Q} M_{J} M_{f} + M_c M_{Q}  M_{H}),
    \end{equation*}
    hence we can conclude that $\norm{\Tmu(x)-x} = 0$. This implies the feasibility of $x$, in the sense that $x = \Tmu(x) \in \X \cap \M$. Then from the expression of $\nabla \psimu$ in Lemma \ref{Le_FBE_gradient}, it holds that 
    \begin{equation*}
        0 \in J(x) \nabla f(x) + \NX(x) \subseteq \nabla f(x) + \mathrm{range}(\Jc(x)) + \NX(x).
    \end{equation*}
    Therefore, $x$ is a first-order stationary point of \eqref{Prob_Ori}. This completes the proof. 
\end{proof}

The equivalence between \eqref{Prob_Ori} and \eqref{Prob_FBE} demonstrates the exactness of our proposed forward-backward semi-envelope. In the rest of this subsection, we aim to establish the relationships between \eqref{Prob_Ori} and \eqref{Prob_FBE} in the aspect of their $\varepsilon$-first-order stationary points. 

\begin{prop}
    \label{Prop_proj_Tx_x_esti}
     Suppose Assumption \ref{Assumption_f} and Assumption \ref{Assumption_Q} hold. Then for any $\mu \leq \mu_{\max}$ and any $x \in  \Y$, it holds that 
    \begin{equation*}
        \norm{(I_n - \Jc(x) \Jc(x)^{\dagger})(\Tmu(x)-x)} \geq \frac{\sigma_{Q}}{ 8M_Q M_c^2 + 2\sigma_{Q} } \norm{\Tmu(x)-x}. 
    \end{equation*}
\end{prop}
\begin{proof}
    For any $x \in \Y$, when $\Tmu(x) - x = 0$, then this proposition holds trivially. Therefore, we only consider the case where $\Tmu(x) - x \neq 0$. 
    
    Let $\lambda = \Jc(x)^{\dagger} (\Tmu(x)-x)$ and $\chi =  \frac{\norm{\Tmu(x)-x - \Jc(x)\lambda}}{\norm{\Tmu(x)-x}}$. Then it holds that 
    \begin{equation}
        \label{Eq_Prop_proj_Tx_x_esti_0}
        \begin{aligned}
            &\norm{\nabla c(x)\tp Q(\Tmu(x))(\Tmu(x)-x - \Jc(x)\lambda)} \\
            \leq{}& \norm{\nabla c(x)\tp Q(\Tmu(x))} \norm{\Tmu(x)-x - \Jc(x)\lambda} \\
            \leq{}& \chi M_{Q} M_{c} \norm{\Tmu(x)-x}. 
        \end{aligned}
    \end{equation}
    Additionally, it follows from Lemma \ref{Le_FBE_error_bound_cQc} that 
    \begin{equation}
        \label{Eq_Prop_proj_Tx_x_esti_1}
        \begin{aligned}
            &\norm{\nabla c(x)\tp Q(\Tmu(x))(\Tmu(x)-x - \Jc(x)\lambda)}\\
            \geq{}& \norm{\nabla c(x)\tp Q(\Tmu(x)) \Jc(x)\lambda} - \norm{\nabla c(x)\tp Q(\Tmu(x))(\Tmu(x)-x)}\\
            \geq{}&\frac{\sigma_{Q}}{4} \norm{\lambda} - 2\mu M_{c} L_{Q} M_{J} M_{f} \norm{\Tmu(x)-x}.
        \end{aligned}
    \end{equation}
    Here, the last inequality follows from \eqref{Eq_Theo_Equivalence_FBE_3}. By combining \eqref{Eq_Prop_proj_Tx_x_esti_0} and \eqref{Eq_Prop_proj_Tx_x_esti_1} together, we arrive at 
    \begin{equation}
        \label{Eq_Prop_proj_Tx_x_esti_2}
        \norm{\lambda} \leq \frac{8\mu M_{c} L_{Q} M_{J} M_{f} + 4\chi M_{Q} M_{c}}{\sigma_{Q}} \norm{\Tmu(x)-x}.
    \end{equation}

    Moreover, it directly follows from the definition of $\chi$ that 
    \begin{equation*}
        \chi \norm{\Tmu(x)-x} = \norm{(\Tmu(x)-x) - \Jc(x)\lambda} \geq \norm{\Tmu(x)-x} - M_c \norm{\lambda}, 
    \end{equation*}
    which leads to the fact that 
    \begin{equation*}
        \begin{aligned}
            &\frac{1-\chi}{M_c} \norm{\Tmu(x)-x} \leq \norm{\lambda} \leq  \frac{8\mu M_{c} L_{Q} M_{J} M_{f} + 4\chi M_{Q} M_{c}}{\sigma_{Q}} \norm{\Tmu(x)-x}\\
            \leq{}& \left(\frac{1}{2M_c} + \frac{4\chi M_Q M_c}{\sigma_{Q}} \right)\norm{\Tmu(x)-x}.
        \end{aligned}
    \end{equation*}
    Here, the second inequality directly follows from \eqref{Eq_Prop_proj_Tx_x_esti_2}, and the third inequality uses the fact that $\mu\leq \frac{\sigma_Q}{16 M_c^2 L_Q M_J M_f}$.
    As a result, we can conclude that 
    \begin{equation*}
        \left(\frac{4M_Q M_c^2}{\sigma_{Q}} + 1\right) \chi \geq \frac{1}{2},
    \end{equation*}
    which illustrates that $\chi \geq \frac{\sigma_{Q}}{ 8M_Q M_c^2 + 2\sigma_{Q} }$. This completes the proof. 
\end{proof}

In the following theorem, we prove that for any $x \in  \Y$ that is an $\varepsilon$-stationary point of \eqref{Prob_FBE}, $\Tmu(x)$ is an $\ca{O}(\varepsilon)$-stationary point of \eqref{Prob_Ori}. 
\begin{theo}
    \label{Theo_Equivalence_FBE_approx}
     Suppose Assumption \ref{Assumption_f} and Assumption \ref{Assumption_Q} hold. Then for any $\mu \leq \mu_{\max}$ and any $x \in \Y$, when $\norm{(I_n - \Jc(x) \Jc(x)^{\dagger})\nabla \psimu(x)} \leq \varepsilon$, it holds that 
     \begin{equation}
        \left\{
         \begin{aligned}
             &\norm{c(\Tmu(x))} \leq \frac{ 16M_Q M_c^3 + 4\sigma_{Q} M_c }{\sigma_{Q}} \mu \varepsilon, \\
             &\mathrm{dist}\left(0,  \nabla f(\Tmu(x)) + \mathrm{range}(\Jc(\Tmu(x))) + \NX(\Tmu(x)) \right) \leq \frac{ 32M_Q M_c^3 + 8\sigma_{Q} M_c }{\sigma_{Q} } \varepsilon.
         \end{aligned}
         \right.
     \end{equation}
\end{theo}
\begin{proof}
    For any $x \in \Y$ that satisfies $\norm{(I_n - \Jc(x) \Jc(x)^{\dagger})\nabla \psimu(x)} \leq \varepsilon$, it holds that 
    \begin{equation*}
        \begin{aligned}
            &\varepsilon \geq \norm{(I_n - \Jc(x) \Jc(x)^{\dagger})\nabla \psimu(x)} \\
            \geq{}& \frac{1}{\mu}\norm{(I_n - \Jc(x) \Jc(x)^{\dagger})(\Tmu(x) - x) - \mu (I_n - \Jc(x) \Jc(x)^{\dagger}) H(x) (\Tmu(x) - x)}\\
            \geq{}& \frac{1}{\mu}\norm{(I_n - \Jc(x) \Jc(x)^{\dagger})(\Tmu(x) - x)} - \norm{(I_n - \Jc(x) \Jc(x)^{\dagger}) H(x) (\Tmu(x) - x)}\\
            \geq{}& \frac{1}{\mu}\cdot \frac{\sigma_{Q}}{ 8M_Q M_c^2 + 2\sigma_{Q} } \norm{\Tmu(x)-x} - M_{H}\norm{\Tmu(x)-x}\\
            \geq{}& \frac{1}{\mu}\cdot \frac{\sigma_{Q}}{ 16 M_Q M_c^2 + 4 \sigma_{Q} } \norm{\Tmu(x)-x}.
        \end{aligned}
    \end{equation*}
    Here the second inequality follows from Lemma \ref{Le_FBE_gradient}, and the fourth inequality uses Proposition \ref{Prop_proj_Tx_x_esti}. Therefore, we can conclude that $ \norm{\Tmu(x)-x} \leq \frac{ 16M_Q M_c^2 + 4\sigma_{Q} }{\sigma_{Q}} \mu\varepsilon$. 
    Hence, the continuity of $c$ implies that 
    \begin{equation*}
        \norm{c(\Tmu(x))} \leq M_c \norm{\Tmu(x)-x} \leq \frac{ 16M_Q M_c^3 + 4\sigma_{Q} M_c }{\sigma_{Q}} \mu \varepsilon.
    \end{equation*}
    Furthermore, we have 
    \begin{equation}
        \label{Eq_Theo_Equivalence_FBE_approx_0}
        \begin{aligned}
            &\mathrm{dist}(0, J(\Tmu(x)) \nabla f(\Tmu(x)) + \NX(\Tmu(x))) \\
            \leq{}& \mathrm{dist}(0, J(x) \nabla f(x) + \NX(\Tmu(x)))  + \norm{J(x) \nabla f(x) - J(\Tmu(x)) \nabla f(\Tmu(x))}\\
            \leq{}& \mathrm{dist}(0, J(x) \nabla f(x) + \NX(\Tmu(x))) + L_H\norm{\Tmu(x) -x}\\
            \leq{}& \frac{1}{\mu} \norm{\Tmu(x) -x} + L_H\norm{\Tmu(x) -x} \\
            \leq{}& \frac{2}{\mu} \norm{\Tmu(x) -x} \leq \frac{ 32M_Q M_c^3 + 8\sigma_{Q} M_c }{\sigma_{Q} } \varepsilon. 
        \end{aligned}
    \end{equation}
    Here the second inequality follows from the definition of $H(x)$ and $L_H$ in Table \ref{Simbol_Definition_2}, and the third inequality uses the optimality condition of the forward-backward semi-envelope subproblem. The second last inequality used the fact that 
    $\mu\leq \frac{1}{4L_H}$.
    Therefore, we can conclude that 
    \begin{equation*}
        \mathrm{dist}\left(0,  \nabla f(\Tmu(x)) + \mathrm{range}(\Jc(\Tmu(x))) + \NX(\Tmu(x)) \right) \leq \frac{ 32M_Q M_c^3 + 8\sigma_{Q} M_c }{\sigma_{Q} } \varepsilon.
    \end{equation*}
    This completes the proof. 
\end{proof}

\subsection{A projected inexact gradient descent method for solving \eqref{Prob_FBE}}
\label{Subsection_algorithm}

In this subsection, we develop 
a projected inexact gradient descent method
for solving \eqref{Prob_FBE}. 

As demonstrated in the equivalence between \eqref{Prob_Ori} and \eqref{Prob_FBE} in Theorem \ref{Theo_Equivalence_FBE}, various existing approaches and their corresponding theoretical analysis can be directly implemented to solve \eqref{Prob_Ori} through \eqref{Prob_FBE}, such as the sequential quadratic programming \cite{curtis2015adaptive}, augmented Lagrangian method \cite{xie2021complexity}, etc.  When employing these methods, we need to compute $\nabla \psimu(x)$, hence requiring the computations of $H(x)$. It is worth mentioning that $H(x)$ has a closed-form expression and its computations can be achieved by successive application of the chain rule. However,  the computations of $H(x)$ involve additional computational costs apart from the computations of $\Tmu(x)$. 
Therefore, we consider the following projected gradient descent method for solving \eqref{Prob_FBE} that avoids the need to compute $H(x)$. Given $x_0\in\M$, at the $k$-th iteration, generate $x_{k+1}$ as follows:  
\begin{equation}
    \label{Eq_FBE_PGD}
    \begin{aligned}
        &g_k ={} \frac{1}{\mu}(I_n - \Jc(x_k)\Jc(x_k)^{\dagger})(\xk - \Tmu(\xk) )\\
        &\xkp ={} \Pi_{\M}(\xk - \eta_k g_k).
    \end{aligned}
\end{equation}
Note that in the above, $I_n - \Jc(x_k)\Jc(x_k)^{\dagger}$
is the projection onto ${\rm null}(\nabla c(x_k)^\top)$, which is the tangent space of $\M$ at $x_k.$ Hence $g_k \in {\rm null}(\nabla c(x_k)^\top)$.

We first present the following proposition illustrating that  $(I_n - \Jc(x)\Jc(x)^{\dagger})(\Tmu(x) - x)$ is a sufficient descent direction for $\psimu(x)$. 
\begin{prop}
    \label{Prop_FBE_P2GD_direction}
    Suppose Assumption \ref{Assumption_f} and Assumption \ref{Assumption_Q} hold. Then for any $\mu \leq \mu_{\max}$ and any $x \in \Y $, it holds that
    \begin{equation*}
        \inner{(I_n - \Jc(x)\Jc(x)^{\dagger})\nabla \psimu(x), \Tmu(x) - x } \leq - \frac{1}{2\mu} \left(\frac{\sigma_{Q}}{8 M_Q M_c^2 +  2\sigma_{Q}} \right)^2\norm{x - \Tmu(x)}^2.
    \end{equation*}
\end{prop}
\begin{proof}
    For any $x \in \Y$, it follows from Lemma \ref{Le_FBE_gradient} that 
    \begin{equation*}
        \begin{aligned}
            &(I_n - \Jc(x)\Jc(x)^{\dagger})\nabla \psimu(x) \\
            ={}& \frac{1}{\mu}(I_n - \Jc(x)\Jc(x)^{\dagger}) (x - \Tmu(x)) -(I_n - \Jc(x)\Jc(x)^{\dagger}) H(x) (x - \Tmu(x)).
        \end{aligned}
    \end{equation*}
    As a result, we have
    \begin{equation*}
        \begin{aligned}
            &\inner{(I_n - \Jc(x)\Jc(x)^{\dagger})\nabla \psimu(x), \Tmu(x) - x } \\
            \leq{}& -\frac{1}{\mu} \inner{x-\Tmu(x), (I_n - \Jc(x)\Jc(x)^{\dagger}) (x-\Tmu(x))} \\
            & + \norm{ H(x) } \norm{x - \Tmu(x)} \norm{(I_n - \Jc(x)\Jc(x)^{\dagger})(x - \Tmu(x))}  \\
            \leq{}& -\frac{1}{\mu} \norm{(I_n - \Jc(x)\Jc(x)^{\dagger})(x - \Tmu(x))}^2 + M_H  \norm{x - \Tmu(x)} \norm{(I_n - \Jc(x)\Jc(x)^{\dagger})(x - \Tmu(x))}\\
            \leq{}&  \frac{\sigma_{Q}}{8 M_Q M_c^2 + 2 \sigma_{Q}} \left(- \frac{1}{\mu} \cdot \frac{\sigma_{Q}}{8 M_Q M_c^2 +  2\sigma_{Q}}  + M_H \right) \norm{x - \Tmu(x)}^2 \\
            \leq{}& - \frac{1}{2\mu} \left(\frac{\sigma_{Q}}{8 M_Q M_c^2 + 2 \sigma_{Q}} \right)^2\norm{x - \Tmu(x)}^2.
        \end{aligned}
    \end{equation*}
    Note that the second last inequality follows from Proposition~\ref{Prop_proj_Tx_x_esti} and the fact that $\mu\leq \frac{\sigma_Q}{M_H(16 M_Q M_c^2+4\sigma_Q)}$.
    This completes the proof. 
\end{proof}

In the following lemma, we estimate an upper bound for $\mathrm{dist}(x+d, \M)$ for any $x \in \Y $ and any $d \in \mathrm{null}(\Jc(x))\tp)$. 
\begin{lem}
    \label{Le_FBE_tangent_dist}
    Suppose Assumption \ref{Assumption_f} and Assumption \ref{Assumption_Q} hold. Then for any $x \in \Y $ and any $d \in \mathrm{null}(\Jc(x)^\top)$ with $\norm{d} \leq \frac{1}{4}\rho$, it holds that 
    \begin{equation*}
        \mathrm{dist}(x+d, \M) \leq \frac{L_c}{\sigma_c}\norm{d}^2.
    \end{equation*}
\end{lem}
\begin{proof}
    For any $x \in  \Y$ and any $d$, it holds that $x+d \in \Z$. Therefore, it holds that 
    \begin{equation*}
        \norm{c(x+d)} \leq \norm{c(x) + \Jc(x)\tp d} + \frac{L_c}{2} \norm{d}^2 = \frac{L_c}{2} \norm{d}^2.
    \end{equation*}
    Furthermore, let $z \in \Pi_{\M}(x + d)$, then similar to the proof  as \cite[Lemma 1]{xiao2023dissolving}, 
    it holds that $\norm{c(x+d)} \geq \frac{\sigma_c}{2} \norm{x+d-z} = \frac{\sigma_c}{2} \mathrm{dist}(x+d, \M)$. 
    As a result, it holds that 
    \begin{equation*}
        \mathrm{dist}(x+d, \M) \leq \frac{2}{\sigma_c}\norm{c(x+d)} \leq \frac{L_c}{\sigma_c} \norm{d}^2. 
    \end{equation*}
    This completes the proof. 
\end{proof}

The next proposition shows that $\psimu(x)$ is greater than
the maximum value of $\psimu(w)$ over $w\in\X$ when $x$ is beyond
the distance of $\frac{\rho}{2}$ from $\X$. 
\begin{prop}   
    \label{Prop_FBE_level_set}
    Suppose Assumption \ref{Assumption_f} and Assumption \ref{Assumption_Q} hold. Then for any $\mu < \mu_{\max}$, and  any $x\in \Y$ such that $\mathrm{dist}(x, \X) \geq \frac{\rho}{2}$, it holds that 
    \begin{equation}
        \psimu(x) > \sup_{w \in \X}\psimu(w). 
    \end{equation}
\end{prop}
\begin{proof}
    For any $x \in \Y $ such that $\mathrm{dist}(x, \X) \geq \frac{\rho}{2}$, 
    it follows from the choices of $\mu$, $\rho$, 
    and $\Tmu(x) = \Pi_\X(x-\mu J(x)\nabla f(x))$, that
    \begin{eqnarray*}
       &&  \norm{\Tmu(x)-x} \leq \norm{x - \Pi_{\X}(x)} +  \norm{\Tmu(x)- \Pi_{\X}(x) } \leq \rho+ M_JM_f \mu 
        \leq \frac{65\rho}{64}
        \\
        && 
        \norm{\Tmu(x)-x} \geq \norm{x - \Pi_{\X}(x)} -  \norm{\Tmu(x)- \Pi_{\X}(x) } \geq \frac{\rho}{2}- M_JM_f \mu 
        \geq \frac{31\rho}{64}.
    \end{eqnarray*}
    As a result, we have
    \begin{equation*}
        \begin{aligned}
            &\psimu(x) = f(x) + \inner{J(x)\nabla f(x), \Tmu(x)-x} + \frac{1}{2\mu} \norm{\Tmu(x)-x}^2 \\
            \geq{}& f(x) -  M_{J} M_{f} \norm{\Tmu(x)-x} + \frac{1}{2\mu} \norm{\Tmu(x)-x}^2 \\
            \geq{}& f(x) - \frac{65\rho}{64}M_{J} M_{f} + \frac{\rho^2}{8 \mu} \geq f(x) + \frac{\rho^2}{16 \mu}.
        \end{aligned}
    \end{equation*}
    Therefore,  for any $x\in \Y$, it holds that 
    \begin{equation*}
        \begin{aligned}
            &\psimu(x) \geq f(x) + \frac{\rho^2}{16\mu} > f(x) + \left(\sup_{w, z \in \W} f(w) - f(z) \right) \\
            \geq{}& f(x) + \left(\sup_{w \in \Z} f(w) - f(x) \right) =  \sup_{w \in \Z} f(w) \geq  \sup_{w \in \X} f(w) \geq \sup_{w \in \X} \psimu(w). 
        \end{aligned}
    \end{equation*}
    This completes the proof. 
\end{proof}

In the following proposition, we  estimate the decrease in $\psimu$ over the iterates $\{\xk\}$, with stepsizes $\{\eta_k\}$ upper-bounded by  
\begin{equation}
    \eta_{\max} = \min\left\{ 
    \frac{\sigma_{Q}^2 \mu}{16(8 M_Q M_c^2 +  2\sigma_{Q})^2}, ~ \frac{\rho \mu}{8M_J M_f L_c} \right\}.
\end{equation}
The proposition shows the descent property of the projected gradient method
\eqref{Eq_FBE_PGD} for minimizing $\psimu$ over $\M$.

\begin{prop}
    \label{Prop_FBE_P2GD_decrease}
    Suppose Assumption \ref{Assumption_f} and Assumption \ref{Assumption_Q} hold. Then for any $\mu \leq \mu_{\max}$,  $x_k \in \Y $ and  $\eta_k \leq \eta_{\max}$, it holds that for $x_{k+1}$ generated by 
    \eqref{Eq_FBE_PGD},
    \begin{equation}
        \psimu(\xkp) \leq \psimu(\xk) - \frac{\eta_k}{4} \cdot \left(\frac{\sigma_{Q}}{8 M_Q M_c^2 +  2\sigma_{Q}} \right)^2 \frac{1}{\mu^2} \norm{\xk - \Tmu(\xk)}^2. 
    \end{equation}
\end{prop}
\begin{proof}
    Let $\yk = \xk - \eta_k g_k$, then it holds that 
    \begin{equation*}
        \begin{aligned}
            &\psimu(\yk) - \psimu(\xk) \leq - \inner{\eta_k g_k, \nabla \psimu(\xk)} + \eta_k^2 \cdot \left(\frac{1}{2} \cdot \frac{2}{\mu}\right) \cdot \frac{1}{\mu^2}\norm{x_k - \Tmu(x_k)}^2\\
            \leq{}& 
            - \frac{\eta_k}{2} \cdot \left(\frac{\sigma_{Q}}{8 M_Q M_c^2 +  2\sigma_{Q}} \right)^2 \frac{1}{\mu^2}\norm{x_k - \Tmu(x_k)}^2 + \eta_k^2 \cdot \frac{1}{\mu^3}\norm{x_k - \Tmu(x_k)}^2.
        \end{aligned}
    \end{equation*}
    Here the first inequality follows from Lemma \ref{Le_FBE_gradient} and the second inequality uses Proposition \ref{Prop_FBE_P2GD_direction}. 

    Moreover, it follows from the Lipschitz smoothness of $c$ that 
    \begin{equation*}
        \begin{aligned}
            &\norm{c(\yk)} \leq \norm{c(\xk) + \inner{\Jc(\xk), \yk - \xk}} + \norm{c(\yk) -  c(\xk) - \inner{\Jc(\xk), \yk - \xk}}\\
            \leq{}& \norm{c(\xk)-\eta_k\inner{\Jc(\xk), g_k}}  + \frac{L_c}{2} \norm{\yk - \xk}^2 = \frac{L_c}{2} \norm{\yk - \xk}^2. 
        \end{aligned}
    \end{equation*}
    Note that the first term in the second inequality is equal to $0$ since
    $x_k\in \M$ and $g_k \in \mathrm{null}(\Jc(x_k)^\top).$
    As a result, by Lemma~\ref{Le_FBE_tangent_dist}, we can conclude that $ \norm{\xkp - \yk} = \mathrm{dist}(\yk, \M) \leq \frac{L_c}{\sigma_c} \norm{\yk - \xk}^2$. 
    Then  it follows from \eqref{Eq_Ub_nabla_psi}  that 
    \begin{equation*}
        \begin{aligned}
            &\psimu(\xkp) - \psimu(\yk) 
        \leq \left(3(M_J+1) M_f + \frac{2\rho}{\mu}\right)\cdot \frac{L_c}{\sigma_c} \norm{\yk - \xk}^2\\
        \leq{}& \eta_k^2 \cdot \left(3(M_J+1) M_f + \frac{2\rho}{\mu}\right)\cdot \frac{L_c}{\sigma_c } \frac{1}{\mu^2}\norm{x_k - \Tmu(x_k)}^2.
        \end{aligned}
    \end{equation*}
    As a result, we have 
    \begin{equation*}
        \begin{aligned}
            \psimu(\xkp) \leq{}&  \psimu(\xk) - \frac{\eta_k}{2} \cdot \left(\frac{\sigma_{Q}}{8 M_Q M_c^2 +  2\sigma_{Q}} \right)^2 \frac{1}{\mu^2} \norm{\xk - \Tmu(\xk)}^2 \\
            &\quad  + \eta_k^2 \cdot  \left(\frac{1}{\mu} + \frac{ (3M_J+1)M_f L_c}{\sigma_c} +\frac{2 \rho L_c}{\sigma_c \mu} \right) \frac{1}{\mu^2} \norm{\xk - \Tmu(\xk)}^2\\
            \leq{}&  \psimu(\xk) - \frac{\eta_k}{2} \cdot \left(\frac{\sigma_{Q}}{8 M_Q M_c^2 +  2\sigma_{Q}} \right)^2 \frac{1}{\mu^2} \norm{\xk - \Tmu(\xk)}^2 \\
            &\quad  + \eta_k^2 \cdot  \left(\frac{5}{4\mu} + \frac{ (3M_J+1)M_f L_c}{\sigma_c}  \right) \frac{1}{\mu^2} \norm{\xk - \Tmu(\xk)}^2\\
            \leq{}& \psimu(\xk) - \frac{\eta_k}{2} \cdot \left(\frac{\sigma_{Q}}{8 M_Q M_c^2 +  2\sigma_{Q}} \right)^2\frac{1}{\mu^2}\norm{\xk - \Tmu(\xk)}^2 + \frac{2\eta_k^2}{\mu}\frac{1}{\mu^2}\norm{\xk - \Tmu(\xk)}^2\\
            \leq{}& \psimu(\xk) - \frac{\eta_k}{4} \cdot \left(\frac{\sigma_{Q}}{8 M_Q M_c^2 +  2\sigma_{Q}} \right)^2 \frac{1}{\mu^2} \norm{\xk - \Tmu(\xk)}^2.
        \end{aligned}
    \end{equation*} 
    Here the second inequality follows from the fact that $\rho\leq \frac{\sigma_c}{8L_c}$. 
    The third follows from the fact that $\mu\leq \frac{\sigma_c}{6(M_J+1)M_f L_c}$,
    which imply that $\frac{ (3M_J+1)M_f L_c}{\sigma_c} \leq \frac{1}{2\mu}$. 
    This completes the proof. 
\end{proof}

Finally, based on Proposition \ref{Prop_FBE_P2GD_decrease}, we present the following theorem illustrating the convergence of \eqref{Eq_FBE_PGD} and its $\ca{O}(\varepsilon^{-2})$ convergence rate. 
\begin{theo}
    \label{Theo_FBE_P2GD_rate}
    Suppose Assumption \ref{Assumption_f} and Assumption \ref{Assumption_Q} hold, and we set $\eta_{\min} \in (0, \eta_{\max}]$.  Then for any $\mu \leq \mu_{\max}$, $x_0 \in \X \cap \M $,  $\eta \in (\eta_{\min}, \eta_{\max})$, and any $K\geq 1$,
   the sequence $\{x_k\}$ generated by \eqref{Eq_FBE_PGD} satisfies the following
   iteration complexity, 
    \begin{equation}
        \begin{aligned}
            &\inf_{k\leq K} \mathrm{dist}\big(0, \nabla f(\Tmu(\xk)) + \mathrm{range}(\nabla c(\Tmu(\xk))) + \NX(\Tmu(\xk)) \big) \\
            \leq{}& \sqrt\frac{\left(32 M_Q M_c^2 +  8\sigma_{Q}\right)^2
            \left( \psimu(x_0) - \inf_{z \in \X}\psimu(z) \right)
            }{\sigma_{Q}^2 \eta_{\min} K}.
        \end{aligned}
    \end{equation}
\end{theo}
\begin{proof}
    Notice that it follows from Proposition \ref{Prop_FBE_P2GD_decrease} that for any $k \geq 0$, we have 
    \begin{equation}
        \label{Eq_Theo_FBE_P2GD_rate_0}
        \psimu(\xkp) \leq \psimu(\xk) - \frac{\eta_k}{4} \cdot \left(\frac{\sigma_{Q}}{8 M_Q M_c^2 +  2\sigma_{Q}} \right)^2 \frac{1}{\mu^2} \norm{\xk - \Tmu(\xk)}^2.
    \end{equation}

    We first prove that $\{\xk\} \subseteq \{x \in \M: \mathrm{dist}(x, \X)\leq \frac{\rho}{2}\}$  by induction. 
    Suppose for some $k \geq 0$, it holds that $\{x_i: 0\leq i \leq k\} \subseteq \{x \in \M: \mathrm{dist}(x, \X)\leq \frac{\rho}{2}\}$. Then Proposition \ref{Prop_FBE_P2GD_decrease} illustrates that $\psimu(\xkp) \leq \psimu(\xk) \leq \psimu(x_0)$. Additionally, it follows from 
    Lemma~\ref{Le_FBE_tangent_dist} and
    Lemma \ref{Le_FBE_Tx_x} that
    \begin{equation*}
       \begin{aligned}
            &\mathrm{dist}(\xkp, \X) \leq \mathrm{dist}(\xk, \X) + \norm{\xkp - \xk}\\
            \leq{}& \frac{\rho}{2} + (\norm{\yk - \xk} + \norm{\xkp - \yk})\\
            \leq{}& \frac{\rho}{2} +   \left( \norm{\yk - \xk}+ \frac{L_c}{\sigma_c} \norm{\yk - \xk}^2 \right)\\
            \leq{}& \frac{\rho}{2} +  2 \norm{\yk - \xk} \leq \frac{\rho}{2} + \frac{2M_J M_f \eta_k}{\mu} \leq \frac{3}{4} \rho < \rho. 
       \end{aligned}
    \end{equation*}
    Here the third and the fourth inequalities follow from the choice of $\eta_{\max}$. 
    Since $\xkp\in \M$ and 
    $\mathrm{dist}(\xkp, \X) \leq \rho$, we
    obtain that  $\xkp\in\Y$. Together with the fact that $\psimu(\xkp) \leq \psimu(x_0)\leq\sup_{w\in\X}\psimu(w)$, Proposition \ref{Prop_FBE_level_set} illustrates that $\mathrm{dist}(\xkp, \X) \leq \frac{\rho}{2}$. As a result, it follows from  induction that $\{\xk\} \subseteq \{x \in \M: \mathrm{dist}(x, \X)\leq \frac{\rho}{2}\}$.

    Then for any $K > 0$, \eqref{Eq_Theo_FBE_P2GD_rate_0} illustrates that 
    \begin{equation*}
        \begin{aligned}
            &\sum_{0\leq k\leq K}  \frac{1}{\mu^2}\norm{\xk - \Tmu(\xk)}^2 \leq \frac{4\left(8 M_Q M_c^2 +  2\sigma_{Q}\right)^2}{\sigma_{Q}^2 \eta_{\min}}  \cdot \left(\psimu(x_0) - \psimu(x_K)\right) \\
            \leq{}& \frac{4\left(8 M_Q M_c^2 +  2\sigma_{Q}\right)^2}{\sigma_{Q}^2 \eta_{\min}} \cdot \left( \psimu(x_0) - \inf_{z \in \X}\psimu(z) \right). 
        \end{aligned}
    \end{equation*}
    As a result, we have 
    \begin{equation*}
        \inf_{k\leq K} \frac{1}{\mu^2}\norm{\xk - \Tmu(\xk)}^2 \leq \frac{4\left(8 M_Q M_c^2 +  2\sigma_{Q}\right)^2\left( \psimu(x_0) - \inf_{z \in \X}\psimu(z) \right)}{\sigma_{Q}^2 \eta_{\min} K} . 
    \end{equation*}
    On the other hand, it follows from the same proof techniques as \eqref{Eq_Theo_Equivalence_FBE_approx_0} that 
    \begin{equation*}
        \begin{aligned}
            &\mathrm{dist}(0, J(\Tmu(x)) \nabla f(\Tmu(x)) + \NX(\Tmu(x)))  \leq \frac{2}{\mu} \norm{\Tmu(x) -x}.
        \end{aligned}
    \end{equation*}
    Therefore, we have 
    \begin{equation*}
        \inf_{k\leq K} \mathrm{dist}(0, \nabla f(\Tmu(\xk)) + \mathrm{range}(\nabla c(\Tmu(\xk))) + \NX(\Tmu(\xk)) )^2 \leq 
        \frac{\left(32 M_Q M_c^2 +  8\sigma_{Q}\right)^2\left( \psimu(x_0) - \inf_{z \in \Y}\psimu(z) \right)}{\sigma_{Q}^2 \eta_{\min} K} . 
    \end{equation*}
    This completes the proof. 
\end{proof}

\subsection{Construction of projective mapping}

In this subsection, we discuss the construction of the projective mapping for \eqref{Prob_Ori} that satisfies Assumption \ref{Assumption_Q}. Although \cite{xiao2025exact} provides some practical schemes for constructing the projective mapping, these schemes cannot guarantee the semi-positive definiteness of $Q(x)$ over $\Rn$. 

We first present the following lemma illustrating how to construct the projective mapping that satisfies Assumption \ref{Assumption_Q} based on the presented projective mappings in \cite{xiao2025exact}. 
\begin{lem}
    Suppose $\widehat{Q}: \Rn \to \bb{S}^{n\times n}$ is locally Lipschitz smooth over $\Rn$, and for any $x \in \X$, it holds that $\mathrm{null}(\widehat{Q}(x)) = \mathrm{range}(\NX(x))$, then the mapping $Q(x) := \widehat{Q}(x)^2$ 
    satisfies all the conditions in Assumption \ref{Assumption_Q}.
\end{lem}
\begin{proof}
    For any $\widehat{Q}: \Rn \to \bb{S}^{n\times n}$ that is locally Lipschitz smooth over $\Rn$, and $\mathrm{null}(\widehat{Q}(x)) = \mathrm{range}(\NX(x))$ for any $x \in \X$, the locally Lipschitz smoothness of $Q$ directly follows from the Lipschitz smoothness of $\widehat{Q}$. Moreover, we can verify that $Q(x) \succeq 0$ holds for any $x \in \Rn$, and thus verifies Assumption \ref{Assumption_Q}(1). Furthermore, as $\mathrm{null}(Q(x)) = \mathrm{null}(\widehat{Q}(x)^2)$ and $\mathrm{null}(\widehat{Q}(x)) = \mathrm{range}(\NX(x))$, we can conclude the validity of Assumption \ref{Assumption_Q}(2). This completes the proof. 
\end{proof}

We end this subsection by presenting some possible choices of the projective mappings $Q$ in Table \ref{Table_Q_mapping} for a variety of commonly encountered $\X$. It can be easily verified that all the projective mappings $Q$ in Table \ref{Table_Q_mapping} satisfy Assumption \ref{Assumption_Q}, and we omit the proofs for simplicity. 

\begin{table}[htb]
    \centering
    \small
    \begin{tabular}{p{4cm}|p{4.5cm}|p{6cm}}
        \hline
        \textbf{Name of constraints} & \textbf{Formulation of $\X$} & \textbf{Possible choices of projective mapping} \\ \hline
        Box constraints &
        $\{x \in \Rn: x\geq 0\}$ &
        $Q(x) = \mathrm{Diag}(x^2)$  \\ \hline
        
        Norm ball &
        $\{x \in \Rn: \norm{x} \leq u\}$ &
        $Q(x) = \frac{1 + \norm{x/u}^4}{2}I_n - \frac{xx\tp}{u^2} $ \\ \hline

        Probability simplex &
        $\{x \in \Rn: x \geq 0, \norm{x}_1 = 1\}$ &
        $Q(x) = \left(\text{Diag}(x) - xx\tp \right)^2$ \\ \hline

        PSD cone &
        $\{X \in \bb{S}^{n\times n} : X \succeq 0\}$ &
        $Q(X) : Y \mapsto \Phi(X^2Y)$ \\ \hline
        
    \end{tabular}
    \caption{Some constraints and their corresponding projective mappings. Here $\Phi$ is the symmetrization of a square matrix, defined as $\Phi(M):= \frac{M+M\tp}{2}$. }
    \label{Table_Q_mapping}
\end{table}

\section{Numerical Experiments}
In this section, we present several illustrative numerical examples to demonstrate the performance of our proposed envelope approach. All experiments are conducted in Python 3.12.2 on a Windows server equipped with an AMD Ryzen 7 5700 CPU and 16 GB of RAM.

To ensure consistency across algorithms with varying stopping criteria, we evaluate the stationarity of the output $x$ using $\mathrm{dist}(0, \nabla f(y) + \NX(y) + \mathrm{range}(\nabla c(y)) )$ with $y = \Pi_{\X}(x)$.  Additionally, we use $\norm{c(\Pi_{\X}(x))}$ to measure the feasibility violation of the output $x$. 

For the envelope parameter $\mu$ in \eqref{Prob_FBE}, we select $\mu = 10^{-3}$ for all test instances. Alternatively, we can employ the adaptive estimation technique \cite{sergeyev2006global} and Monte Carlo techniques \cite{marion2023finite,xiao2025cdopt} to estimate the upper or lower bounds of the constants in Tables \ref{Simbol_Definition}-\ref{Simbol_Definition_2}, hence providing an estimate for $\mu_{\max}$.

\subsection{Semi-positive definite cone with spherical constraints}
\label{Subsection_Numerocal_experiment_1}
In this subsection, we test the performance of our proposed envelope approach on the following optimization problem over the semi-positive definite (SPD) cone,
\begin{equation}
\label{Example_SDP}
\begin{aligned}
\min_{X \in \bb{S}^{n \times n}} \quad &  \inner{B, X} + \frac{1}{2} \inner{X, \ca{H}(X)} + \frac{\nu}{6} \norm{X}\ff^3\
 \quad \text{s. t.} &\norm{X}\fs = 1, \quad  X \succeq 0, \quad \norm{X}_2 \leq M.
\end{aligned}
\end{equation}
Here $B \in \bb{S}^{n \times n}$ is symmetric matrix, $\ca{H}: \bb{S}^{n \times n} \to \bb{S}^{n \times n}$ is a self-adjoint linear mapping over $\bb{S}^{n \times n}$,  $\nu > 0$ is the regularization parameter for the cubic term, and we set $M = 10^{6}$.
For this problem, we have $\X = \{X\in \bb{S}^{n\times n}: \norm{X}_2 \leq M, \; X \succeq 0\}$, and $\M = \{X\in\bb{S}^{n\times n}: \norm{X}\fs = 1\}$. The corresponding constraint dissolving mapping for $\X$ is chosen as 
\begin{equation}
    Q(X) = Y \mapsto \Phi(X^2 \Theta_M(X)^2 Y).
\end{equation}
Here we define 
\begin{equation}
    \theta_M(x) =
\begin{cases}
1, & x \leq M-1 \\
2(x - (M - 1) )^3 - 3(x - (M    - 1) )^2 + 1, & M-1 < x \leq M \\
0, & x > M,
\end{cases}
\end{equation}
and for any symmetric matrix $X = U \mathrm{Diag}(\lambda) U\tp$ with $\lambda \in \Rn$, we define 
\begin{equation}
    \Theta_{M}(X) = U \mathrm{Diag}(\theta_{M}(\lambda)) U\tp.
\end{equation}
It is worth mentioning that for any $X$ such that $\norm{X}_2 \leq M-1$, it holds that $Q(X) = Y \mapsto \Phi(X^2 Y)$.

Based on the partial forward-backward envelope \eqref{Prob_FBE}, we reformulate the constrained optimization problem \eqref{Example_SDP} as the following equality-constrained problem,
\begin{equation}
\label{Example_SDP_FBE}
\begin{aligned}
\min_{X \in \bb{R}^{n \times n}} \quad &\psimu(X)\
\quad \text{s. t.}  & \norm{X}\fs = 1.
\end{aligned}
\end{equation}
Here $\psimu$ is the forward-backward semi-envelope for \eqref{Example_SDP}. 
For the above nonlinear nonconvex matrix optimization problem \eqref{Example_SDP}, there are limited choices of efficient and easy-to-use solvers. While PENLAB can handle \eqref{Example_SDP}, it is only {\sc Matlab}-based. To our knowledge, no Python-based solver exists for \eqref{Example_SDP}.

For all test instances, we first randomly generate a symmetric matrix $T \in \bb{R}^{n^2 \times n^2}$ with entries drawn from the standard normal distribution and then normalize it by $\norm{T}$. We define $\ca{H}:\bb{S}^{n\times n} \to \bb{S}^{n\times n}$ as $\ca{H}(X) = \Phi(\mathrm{reshape}(T\mathrm{vec}(X)))$ for any $X \in \bb{S}^{n\times n}$, where $\mathrm{vec}(X)$ is the column-wise vectorization of $X$, and $\mathrm{reshape}(\cdot)$ is its inverse (i.e., $\mathrm{reshape}(\mathrm{vec}(X)) = X$). This ensures $\ca{H}$ is self-adjoint over 
$\bb{S}^{n \times n}$. We set $\nu = 1.0$ and $\mu = 0.01$ in \eqref{Prob_FBE}  for all instances.  For all solvers, the tolerance is set to $10^{-5}$, and the maximum runtime is chosen as $300$ seconds. For each instance, we generate $B$ and $\hat{X}_0$ via i.i.d. normal sampling, and then project $\hat{X}_0$ onto the PSD cone to obtain the initial point $X_0$.

We compare our projected gradient descent method (PGD) in \eqref{Eq_FBE_PGD} with several efficient solvers. Treating \eqref{Example_SDP_FBE} as an equality-constrained problem, we use SLSQP and TRCON from SciPy. Alternatively, by viewing \eqref{Example_SDP_FBE} as an optimization problem over the spherical manifold, we apply the Riemannian gradient descent (RGD) method and Riemannian conjugate gradient (RCG) method from PyManopt. For our PGD, we use Barzilai-Borwein adaptive stepsizes and non-monotone line search.

The numerical results are shown in Table \ref{Table_SDP}. From these numerical results, we verify the equivalence between \eqref{Prob_Ori} and \eqref{Prob_FBE}. Moreover, these numerical results demonstrate that Riemannian solvers from PyManopt fail to achieve high accuracy, while PGD and TRCON yield solutions with high accuracy. Moreover, our proposed PGD method is significantly faster than TRCON, which further demonstrates the potential of our forward-backward envelope and the proposed projected gradient descent method in \eqref{Eq_FBE_PGD}.

    \begin{table}[tb]
		\begin{center}
			\footnotesize
			\begin{minipage}{\textwidth}
				\caption{Numerical experiments on solving \eqref{Example_SDP} through \eqref{Example_SDP_FBE}.}
				\label{Table_SDP}
				\begin{tabular*}{\textwidth}{c@{\extracolsep{\fill}}ccccccc@{\extracolsep{\fill}}}
					\toprule \midrule
					Test instance& Solver &
					Function value & Iterations & Function evaluations & Stationarity & Feasibility
					& CPU time (s) \\
                    \hline
                        \multirow{5}{*}{$n={10}$}
&  PGD   & -9.446e-01 & 94 & 97 & 5.435e-06 & 0.000e+00 & 0.218 \\
&  RGD   & -9.663e-01 & 65 & 65 & 1.207e-01 & 8.476e-02 & 0.308 \\
&  RCG   & -1.034e+00 & 43 & 43 & 1.438e-01 & 3.685e-01 & 0.284 \\
&  TRCON & -9.446e-01 & 86 & 96 & 1.525e-05 & 9.864e-11 & 0.483 \\
&  SLSQP & -9.688e-01 & 1000 & 10705 & 2.027e-01 & 7.723e-02 & 11.154 \\
                         \hline
                         \multirow{5}{*}{$n={20}$}
&  PGD   & -1.658e+00 & 94 & 97 & 8.917e-06 & 2.220e-16 & 0.231 \\
&  RGD   & -1.735e+00 & 76 & 76 & 1.790e-01 & 1.294e-01 & 0.325 \\
&  RCG   & -2.009e+00 & 41 & 41 & 2.307e-01 & 6.936e-01 & 0.243 \\
&  TRCON & -1.658e+00 & 76 & 86 & 4.922e-05 & 1.644e-12 & 0.728 \\
&  SLSQP & -1.720e+00 & 1000 & 10667 & 2.616e-01 & 9.679e-02 & 45.659 \\
                        \hline
                        \multirow{5}{*}{$n={30}$}
&  PGD   & -1.847e+00 & 84 & 86 & 4.833e-06 & 4.441e-16 & 0.351 \\
&  RGD   & -2.002e+00 & 71 & 71 & 1.785e-01 & 2.267e-01 & 0.501 \\
&  RCG   & -2.243e+00 & 41 & 41 & 2.562e-01 & 6.680e-01 & 0.419 \\
&  TRCON & -1.847e+00 & 65 & 68 & 8.923e-05 & 3.127e-11 & 1.299 \\
&  SLSQP & - & - & - & - & - & $>300$ \\
                        \hline
                        \multirow{5}{*}{$n={50}$}
&  PGD   & -2.323e+00 & 91 & 93 & 5.830e-06 & 2.220e-16 & 1.082 \\
&  RGD   & -2.581e+00 & 53 & 53 & 2.008e-01 & 2.905e-01 & 1.251 \\
&  RCG   & -2.979e+00 & 39 & 39 & 2.651e-01 & 8.676e-01 & 1.323 \\
&  TRCON & -2.323e+00 & 67 & 77 & 1.216e-04 & 9.163e-11 & 31.039 \\
&  SLSQP & - & - & - & - & - & $>300$ \\
                        \hline
					\bottomrule
				\end{tabular*}
			\end{minipage}
		\end{center}
	\end{table}

\subsection{Semi-positive definite cone with linear constraints}
In this subsection, we test the performance of our proposed envelope approach on the following optimization problem over the SPD cone with additional linear constraints, 
\begin{equation}
    \label{Example_SDP_2}
    \begin{aligned}
        \min_{X \in \bb{R}^{n \times n}} \quad &  \inner{B_0, X} + \frac{1}{2} \inner{X, \ca{H}(X)} + \frac{\nu}{6} \norm{X}\ff^3\\
        \text{s. t.} \quad &\ca{B}(X) = b, \quad  X \in \X.
    \end{aligned}
\end{equation}
Here $\X = \{X\in \bb{S}^{n\times n}: \norm{X}_2 \leq M, \; X \succeq 0\}$ with $M = 10^{6}$,  $B_0 \in \bb{S}^{n \times n}$ is a square matrix, $\ca{H}: \bb{S}^{n \times n} \to \bb{S}^{n \times n}$ is a self-adjoin linear mapping, 
$\ca{B}: \bb{S}^{n \times n} \to \Rm$ is a linear mapping,  and $\nu > 0$ is the regularization parameter for the cubic term. 
Then based on the forward-backward semi-envelope \eqref{Prob_FBE}, we reformulate the constrained optimization problem \eqref{Example_SDP_2} into the following equality-constrained optimization problem,
\begin{equation}
    \label{Example_SDP_2_FBE}
    \begin{aligned}
        \min_{X \in \bb{R}^{n \times n}} \quad &\psimu(X)\\
        \text{s. t.} \quad & \ca{B}(X) = b.
    \end{aligned}
\end{equation}
Here $\psimu$ is the forward-backward semi-envelope for \eqref{Example_SDP_2}. 
It is worth mentioning that  it is generally expensive to compute
the projection onto the 
feasible region of \eqref{Example_SDP_2} since it 
generally does not have a closed-form solution. 
On the other hand, \eqref{Example_SDP_2_FBE} is a nonlinear optimization over the affine set $\M=\{X\in \bb{R}^{n\times n}:\ca{B}(X) = b\}$, which can be efficiently solved by various existing solvers. 

For all the test instances in this subsection, we generated the matrix $B$ and the mapping $\ca{H}$ in the same technique as those in Section \ref{Subsection_Numerocal_experiment_1}, while $b \in \Rm$ is randomly generated by the ``randn'' function in the NumPy package. Moreover, for the linear mapping $\ca{B}$, we randomly generate 
matrices  $B_1,\ldots, B_m \in \bb{S}^{n \times n}$ with entries drawn from the standard normal distribution, and define the linear mapping $\ca{B}$ by $\ca{B}(X) = (\inner{B_i,X})_{i=1}^m$ for any $X \in \bb{S}^{n\times n}$. Additionally, we set $\nu = 1.0$ for all instances, and choose $\mu = 0.001$ in the forward-backward envelope \eqref{Prob_FBE}. 

In our numerical experiments, we compare the efficiency of our proposed projected gradient descent method (PGD) \eqref{Eq_FBE_PGD} with TRCON and SLSQP. It is worth mentioning that although the affine space is a smooth manifold, it is not supported in the PyManopt package. Therefore, the Riemannian optimization solvers, such as RGD and RCG, cannot be applied to solve \eqref{Example_SDP_2_FBE}. The results of our numerical experiments are exhibited in Table \ref{Table_SDP}. From these presented results, we can observe that both PGD and TRCON can successfully solve the optimization problem \eqref{Example_SDP_2_FBE}, while the SLSQP solver fails in each test instance. Moreover,  the PGD solver achieves superior performance over the TRCON solver. These numerical experiments further demonstrate the great potential of our proposed forward-backward envelope approach and  the projected gradient method \eqref{Eq_FBE_PGD}.

    \begin{table}[tb]
		\begin{center}
			\footnotesize
			\begin{minipage}{\textwidth}
				\caption{Numerical experiments on solving SDP problems through \eqref{Example_SDP_2}.}
				\label{Table_SDP_2}
				\begin{tabular*}{\textwidth}{c@{\extracolsep{\fill}}cccccccc@{\extracolsep{\fill}}}
					\toprule \midrule
					Test instance& Solver &
					Function value & Iterations & Function evaluations & Stationarity & Feasibility
					& CPU time (s) \\
                    \hline
                    \multirow{3}{*}{$n={10}$}
&  PGD   & 1.090e+03 & 97 & 98 & 3.604e-06 & 8.555e-13 & 0.205 \\
&  TRCON & 1.090e+03 & 204 & 208 & 1.289e-05 & 1.158e-12 & 0.893 \\
&  SLSQP & -2.791e+32 & 4 & 4 & 7.471e+17 & 4.274e+08 & 0.020 \\
                    \hline 
                    \multirow{3}{*}{$n={20}$}
&  PGD   & 3.330e+03 & 274 & 288 & 7.954e-06 & 4.433e-13 & 0.811 \\
&  TRCON & 3.330e+03 & 510 & 511 & 3.451e-05 & 1.592e-12 & 6.337 \\
&  SLSQP & -1.030e+47 & 7 & 18 & 4.538e+25 & 1.963e+14 & 0.639 \\
                    \hline 
                    \multirow{3}{*}{$n={30}$}
&  PGD   & 2.936e+04 & 173 & 178 & 7.645e-06 & 1.775e-12 & 3.350 \\
&  TRCON & 2.935e+04 & 349 & 349 & 8.346e-05 & 7.227e-11 & 19.249 \\
&  SLSQP & -8.502e+35 & 5 & 15 & 9.220e+19 & 1.187e+12 & 5.018 \\
                    \hline 
                    \multirow{3}{*}{$n={50}$}
&  PGD   & 8.555e+04 & 262 & 269 & 6.413e-06 & 5.687e-12 & 7.197 \\
&  TRCON & - & - & - & - & - & $>300$ \\
&  SLSQP & - & - & - & - & - & $>300$ \\
                    \hline 
					\bottomrule
				\end{tabular*}
			\end{minipage}
		\end{center}
	\end{table}

\section{Conclusion}

In this paper, we consider the constrained optimization problem of the form \eqref{Prob_Ori} subject to equality constraints $c(x) = 0$ and 
$x\in\X$.
The nonconvexity of $c(x)$ in the problem leads to the nonconvexity of its feasible region $\K$, hence presenting significant challenges in the development of envelope-based approaches for solving \eqref{Prob_Ori}.

To overcome these difficulties, inspired by the forward-backward envelope approaches for convex-constrained optimization problems,  we introduce the forward-backward semi-envelope approach for \eqref{Prob_Ori}. Specifically, we reformulate \eqref{Prob_Ori} as the minimization of a continuously differentiable objective function $\psimu$ with the equality constraints $c(x) = 0$. We establish the equivalence between \eqref{Prob_Ori} and \eqref{Prob_FBE} in the aspect of their first-order stationary points. 

Based on the established equivalence between \eqref{Prob_Ori} and \eqref{Prob_FBE}, a wide range of existing optimization techniques designed for equality-constrained problems, along with their associated theoretical guarantees, can be directly utilized to address \eqref{Prob_Ori}. Furthermore, motivated by the properties of the approximate gradient of $\psimu(x)$, we develop a projected gradient method specifically for solving \eqref{Prob_Ori}. We establish its convergence properties and prove the $\ca{O}(\varepsilon^{-2})$ worst-case iteration complexity.

Preliminary numerical experiments further substantiate the effectiveness and potential of our proposed methods in solving \eqref{Prob_Ori}. These results indicate that the semi-envelope framework offers a promising direction for future research in nonconvex constrained optimization.

\bibliographystyle{plain}
\bibliography{ref}

\end{document}